\documentclass[11pt]{article}

\makeindex
\usepackage[all]{xy}
\usepackage{graphicx}

\usepackage{amsmath}
\usepackage{amsfonts}
\usepackage{amssymb}
\usepackage{amsthm}

\newcommand{\isom}{\cong}

\newcommand{\Z}{{\bf{Z}}}
\newcommand{\N}{{\bf{N}}}
\newcommand{\Q}{{\bf{Q}}}
\newcommand{\Qbar}{{\overline{\Q}}}
\newcommand{\R}{{\bf{R}}}
\newcommand{\C}{{\bf{C}}}
\newcommand{\F}{{\bf{F}}}
\newcommand{\T}{{\bf{T}}}

\newcommand{\m}{{\mathfrak{m}}}
\newcommand{\PP}{\bf{P}}

\newcommand{\Mid}{|} 
\newcommand{\lm}{\Big|}
\newcommand{\miD}{|}
\newcommand{\rmid}{\Big|}
\newcommand{\divs}{\!\mid\!}
\newcommand{\ndiv}{\!\nmid\!}
\newcommand{\tensor}{\otimes}

\newcommand{\ra}{{\rightarrow}}
\newcommand{\lra}{{\longrightarrow}}

 1
\DeclareFontEncoding{OT2}{}{} 
  \newcommand{\textcyr}[1]{%
    {\fontencoding{OT2}\fontfamily{wncyr}\fontseries{m}\fontshape{n}%
     \selectfont #1}}
\newcommand{\Sha}{{\mbox{\textcyr{Sh}}}} 

\newcommand{\caf}{c_{\scriptscriptstyle{A_f}}}

\newcommand{\If}{{I_{\scriptscriptstyle{f}}}}

\newcommand{\Afdual}{A_f^{\vee}}

\newcommand{\Hom}{{\rm Hom}}
\newcommand{\Ann}{{\rm Ann}}
\newcommand{\End}{{\rm End\ }}
\newcommand{\im}{{\rm im}}
\newcommand{\rank}{{\rm rk}}

\newcommand{\qq}{{\mathfrak{q}}}

\newcommand{\Frob}{{\rm Frob}}
\newcommand{\ord}{{\rm ord}}
\newcommand{\LAf}{{L_{\scriptscriptstyle{A_f}}}}
\newcommand{\OAf}{{\Omega_{\scriptscriptstyle{A_f}}}}

\newcommand{\NerA}{{\mathcal{A}}}
\newcommand{\Tl}{{T_{\scriptscriptstyle{\ell}}}}
\newcommand{\Vl}{{V_{\scriptscriptstyle{\ell}}}}
\newcommand{\Wl}{{W_{\scriptscriptstyle{\ell}}}}
\newcommand{\Zl}{{\Z_\ell}}
\newcommand{\Ql}{{\Q_\ell}}
\newcommand{\Ip}{{I_p}}
\newcommand{\Qp}{{\Q_p}}
\newcommand{\Frobp}{{{\rm Frob}_p}}

\newcommand{\IA}{{I_{\scriptscriptstyle{A'}}}}
\newcommand{\IB}{{I_{\scriptscriptstyle{B'}}}}
\newcommand{\IC}{{I_{\scriptscriptstyle{C'}}}}
\newcommand{\SA}{{S_{\scriptscriptstyle{A'}}}}
\newcommand{\SB}{{S_{\scriptscriptstyle{B'}}}}
\newcommand{\SC}{{S_{\scriptscriptstyle{C'}}}}
\newcommand{\ICp}{{\IC}}
\newcommand{\TCp}{{\T_{\scriptscriptstyle{C'}}}}
\newcommand{\RCp}{{R_{\scriptscriptstyle{C'}}}}
\newcommand{\SCp}{{S_{\scriptscriptstyle{C'}}}}
\newcommand{\Shan}{{\Mid \Sha(A_f) \miD_{\rm an}}}
\newcommand{\Shaf}{{\Mid \Sha(A_f) \miD}}

\newcommand{\dirlimn}{{\displaystyle \lim_{\stackrel{\lra}{n}}}\ }
\newcommand{\dirlimm}{{\displaystyle \lim_{\stackrel{\lra}{m}}}\ }

\newcommand{\comment}[1]{}
\newcommand{\marginalfootnote}[1]{%
   \footnote{#1}\marginpar{\hfill {\sf\thefootnote}}%
}

\newcommand{\edit}[1]{\marginalfootnote{#1}}
\DeclareMathOperator{\tor}{tor}

\newtheorem{lem}{Lemma}[section]
\newtheorem{cor}[lem]{Corollary}
\newtheorem{prop}[lem]{Proposition}
\newtheorem{conj}[lem]{Conjecture}
\newtheorem{thm}[lem]{Theorem}

\theoremstyle{definition}

\newtheorem{rmk}[lem]{Remark}

\newtheorem{eg}[lem]{Example}

\newcommand{\thetitle}
{A visible factor of the special L-value}

\begin{document}
\parindent=2em

\title{\thetitle}
\author{Amod Agashe\footnote{The author was partially supported by National Science Foundation Grant No. 0603668. Mathematics subject classification: 11G40.}
}
\maketitle

\begin{abstract}
Let~$A$ be a quotient of~$J_0(N)$ associated to a newform~$f$
such that the special $L$-value of~$A$ (at $s=1$) is non-zero.
We give a formula for
the ratio of 
the special $L$-value to the real period of~$A$ that expresses
this ratio as a rational number.
We extract an integer factor from the numerator of this formula; 
this factor
is non-trivial in general and is related to certain congruences
of~$f$ with eigenforms of positive analytic rank. 
We use the techniques of visibility to show that,
under certain hypotheses (which includes 
the first part of the Birch and Swinnerton-Dyer 
conjecture on rank), if an odd prime~$q$
divides this factor, then $q$~divides either
the order of the Shafarevich-Tate group or the order
of a component group of~$A$. 
Suppose $p$ is an odd prime such that $p^2$ does not
divide~$N$, 
$p$ does not divide the order of the rational torsion subgroup
of~$A$, and
$f$ is congruent modulo a prime ideal over~$p$
to an eigenform whose
associated abelian variety has positive Mordell-Weil rank.
Then we show that $p$ divides the factor mentioned above;
in particular, $p$ divides the numerator of
the ratio of 
the special $L$-value to the real period of~$A$. 
Both of these results are as implied by 
the second part of the Birch and Swinnerton-Dyer conjecture, and 
thus provide theoretical evidence towards the conjecture.
\comment{
Let~$A$ be a quotient of~$J_0(N)$ associated to a newform~$f$
such that the special $L$-value of~$A$ (at $s=1$) is non-zero.
We extract an integer factor from the ratio of 
the special $L$-value to the real period of~$A$; this factor
is non-trivial in general and is related to certain congruences
of~$f$ with eigenforms of positive analytic rank. 
We use the techniques of visibility to show that,
assuming the first part of the Birch and Swinnerton-Dyer 
conjecture on rank, and 
under some other mild hypotheses, if an odd prime~$q$
divides this factor, then $q$~divides either
the order of the Shafarevich-Tate group or the order
of a component group of~$A$. 
This is as predicted by 
the second part of the Birch and Swinnerton-Dyer conjecture, and 
thus is a partial result towards it.

Our approach may
also be viewed as a way of linking the two parts of the conjecture.
If the level~$N$ is prime,
then under a reasonable hypothesis on special $L$-values of
certain twists of~$f$, we show that if an odd non-Eisenstein prime~$q$ 
divides the factor mentioned above, then $q^2$ divides the 
ratio of the special $L$-value of~$A$ to the period mentioned above.
Both the statements above are}

\end{abstract}

\section{Introduction} \label{section:intro}

Mazur introduced the notion of visibility in order to better 
understand geometrically the elements of the 
Shafarevich-Tate group of an abelian variety.
The corresponding theory, which we call the theory of visibility, 
has been used to show the existence of 
non-trivial elements of the Shafarevich-Tate group
of abelian varieties and motives (e.g., see~\cite{cremona-mazur},
\cite{agst:vis}, \cite{dsw}). 
The second part of the
Birch and Swinnerton-Dyer conjecture gives a 
formula for the order of the Shafarevich-Tate group,
and one might wonder how much of this conjectural order
can be explained by the theory of visibility, as well
as whether, when the theory of visibility implies that 
the Shafarevich-Tate group is non-trivial, the Birch and Swinnerton-Dyer
conjectural order of the Shafarevich-Tate group is non-trivial as well.
While this issue has been investigated computationally
(e.g., see~\cite{cremona-mazur}, \cite{agst:bsd}),
to our knowledge 
there have been no general theoretical results that directly link visibility
to the Birch and Swinnerton-Dyer conjecture. 

In this paper, for abelian subvarieties of~$J_0(N)$ 
associated to newforms and having analytic rank zero,
we take the first step in theoretically linking visibility
to the second part of the Birch and Swinnerton-Dyer conjecture.
Firstly, when the newform associated to such an abelian variety
is congruent to an eigenform whose associated abelian variety has
positive Mordell-Weil rank,
the theory of visibility shows that
the actual Shafarevich-Tate group of such an abelian variety is non-trivial
(subject to some other minor hypotheses).
We prove that in  this situation, under some additional hypotheses,
the Birch and Swinnerton-Dyer conjectural order of
the Shafarevich-Tate group is non-trivial as well.
Secondly, following an idea of L.~Merel, we  
extract an integer factor from the Birch and Swinnerton-Dyer
conjectural formula for the order
of the Shafarevich-Tate group.
This factor is a variant of the notion of the modular degree
and is divisible by certain congruence primes associated to the newform.
We show, using the theory of visibility, that if an odd prime divides
this factor, then it also divides the actual order
of the Shafarevich-Tate group 
under certain assumptions
(the strongest assumption being the first part of
the Birch and Swinnerton-Dyer formula on ranks
of Mordell-Weil groups).
We now describe our results more precisely.
\comment{
This may be seen as a first step in an effort to use the theory of visibility
to show that the Birch and Swinnerton-Dyer
conjectural order of the Shafarevich-Tate group divides the actual order.
In any case, both of the results above provide new theoretical evidence for
the second part of
the Birch and Swinnerton-Dyer conjecture in the analytic rank zero case.}

Let $N$ be a positive integer. Let
$X_0(N)$ denote the modular curve over~$\Q$ associated to~$\Gamma_0(N)$,
and let $J_0(N)$ be its Jacobian, which is an abelian variety
over~$\Q$.
Let $f$ be a newform of weight~$2$ on~$\Gamma_0(N)$. 
We will denote by~$L(f,s)$ the $L$-function associated to~$f$.
Let $\T$ denote the subring of endomorphisms of~$J_0(N)$
generated by the Hecke operators (usually denoted~$T_\ell$
for $\ell \ndiv N$ and $U_p$ for $p\divs N$). 
Let $I_f = {\rm Ann}_{\T} f$.
Let $A_f$ denote the quotient abelian variety $J_0(N)/ I_f J_0(N)$
over~$\Q$, which was introduced by Shimura in~\cite{shimura:factors}. 
If $f$ has integer Fourier coefficients, then $A_f$ is just an elliptic curve,
and the reader may assume this for simplicity; in any casse, in the proof of our
main results, the dimension of~$A_f$ does not play a major role.
Let $\LAf(s)$ denote the $L$-function associated to~$A_f$.

The order of vanishing of $\LAf(s)$ at $s=1$ is called the
{\it analytic rank} of~$A_f$, 
and the order of vanishing of $L(f,s)$ at $s=1$ is called the
{\em analytic rank} of~$f$.
The rank of the finitely generated abelian group~$A_f(\Q)$ is called
the {\em Mordell-Weil rank} of~$A_f$.
The first part of the Birch and Swinnerton-Dyer conjecture is the following:

\begin{conj}[Birch and Swinnerton-Dyer] \label{bsd1}
The {\it analytic rank} of~$A_f$ is equal to its Mordell-Weil rank.
\end{conj}

Now suppose that $\LAf(1) \neq 0$. Then 
by~\cite{kollog:finiteness}, $A_f$ has Mordell-Weil rank zero,
and the Shafarevich-Tate group~$\Sha(A_f)$ of~$A_f$ is finite.
Let $\NerA$ denote the N\'eron model of~$A_f$ over~$\Z$
and let $\NerA^0$ denote the largest open subgroup scheme
of~$\NerA$ in which all the fibers are connected.
Let $d = \dim A_f$, and let $D$ be a generator of
the $d$-th exterior power of the group of invariant
differentials on~$\NerA$.
Let $\OAf$ denote the volume of~$A_f(\R)$ with
respect to the measure given by~$D$.
If $p$ is a prime number, then the group of~$\F_p$-valued
points of the quotient~$\NerA_{\F_p}/\NerA^0_{\F_p}$ is called
the (arithmetic) component group of~$A$ and its order
is denoted $c_p(A)$.
If $A$ is an abelian variety, then we denote by~$A^{\vee}$
the dual abelian variety of~$A$. Throughout this article, we use the 
symbol~$\stackrel{?}{=}$ to denote a conjectural equality.

Considering that $\LAf(1) \neq 0$,
the second part of the Birch and Swinnerton-Dyer conjecture says the following:

\begin{conj}[Birch and Swinnerton-Dyer] \label{bsd2}
\begin{eqnarray} \label{bsdform}
\frac{\LAf(1)}{\OAf} 
\stackrel{?}{=}
\frac {\Mid \Sha(A_f) \miD \cdot \prod_{\scriptscriptstyle{p |N}}  c_p(A_f)}
      { \Mid A_f(\Q) \miD \cdot \Mid A_f^{\vee}(\Q) \miD }.
\end{eqnarray}
\end{conj}

It is known that ${\LAf(1)}/{\OAf}$ is a rational number
(this also follows from Theorem~\ref{thm:lovero} below) and
we call this number the algebraic part of the special
$L$-value of~$A_f$. 
The importance of the second part of
the Birch and Swinnerton-Dyer conjecture is that it gives a conjectural
value of~$\Mid \Sha(A_f) \miD$ in terms of the other quantities
in~(\ref{bsdform}) (which can often be computed). Let us denote this 
conjectural value of~$\Mid \Sha(A_f) \miD$ 
by~$\Mid \Sha(A_f) \miD_{\rm an}$ (where ``an'' stands for ``analytic'').
The theory of
Euler systems has been used to bound~$|\Sha(A_f)|$ 
from above in terms~$\Shan$ 
as in the work
of Kolyvagin and of Kato (e.g., see \cite[Thm $8.6$]{rubin:eulerell}).
Also, the Eisenstein series method 
is being used by Skinner-Urban (as yet unpublished)
to try to show that $\Shan$ divides~$|\Sha(A_f)|$.
In both of the methods above,
one may have to stay away from certain primes. 

\comment{
As mentioned earlier, in
this paper we link the theory of visibility to the second part
of the Birch and Swinnerton-Dyer conjecture. This is achieved 
in two related, but different ways, as we now explain.
}
In order to explain our results,
for the sake of simplicity, we assume for the moment that $N$ is prime.
Suppose $q$ is a prime such that $f$ is congruent modulo 
a prime ideal over~$q$ to a newform~$g \in S_2(\Gamma_0(N), \C)$ 
such that~$A_g$ has positive Mordell-Weil rank. 
Then the theory of visibility often implies the existence of
non-trivial elements of~$\Sha(A_f)$ of order~$q$ (for example,
this happens if $q \ndiv N(N-1)$ -- see
the proof of Theorem~\ref{thm:mainprime} below; this
kind of result is well known).  
Now if the second part of the Birch and Swinnerton-Dyer conjecture
is true, then $q$ should also divide~$\Shan$. We show that this is indeed
the case:

\begin{thm} \label{thm:first}
Recall that we are assuming that $N$ is prime. 
Suppose $q$ is an odd prime such that $f$ is congruent modulo 
a prime ideal over~$q$ to a newform~$g \in S_2(\Gamma_0(N), \C)$ 
such that~$A_g$ has positive Mordell-Weil rank. 
Then $q$ divides~$\Shan$.
\end{thm}
The theorem follows from Corollary~\ref{cor:mwrk} below.
Thus, in certain situations, when the theory of visibility
implies that~$\Shaf$ is non-trivial, we show that $\Shan$ is 
non-trivial as well, which
provides
theoretical evidence towards the second part of the Birch
and Swinnerton-Dyer conjecture based on the theory of visibility.

Our other contribution towards the second
part of the Birch and Swinnerton-Dyer conjecture is that we
extract a factor of~$\Shan$ that can be related to~$|\Sha(A_f)|$,
as we now explain.
Let $H=H_1(J_0(N),\Z)$ and let $\pi_*$ denote the map
$H_1(J_0(N),\Q) \ra H_1(A_f,\Q)$ induced by the quotient
map $\pi: J_0(N) \ra A_f$. Let $K$ denote the kernel of~$\pi_*$ restricted to~$H$. 
Let $e \in H_1(J_0(N),\Q)$
denote the winding element (whose definition is recalled below in 
Section~\ref{section:formula})
and let $I_e$ denote the annihilator of~$e$ under
the action of~$\T$.
Let $\Im$ denote the annihilator, under the action of~$\T$,
of the divisor $(0) - (\infty)$,
considered as an element of~$J_0(N)(\C)$. 
There is a complex conjugation involution acting on~$H$,
and if $G$ is a group on which it induces an involution,
then by $G^+$ we mean the subgroup of elements of~$G$ fixed
by the involution. By Lemma~\ref{lem:Ie} below, $\Im e \subseteq H[I_e]^+$.
We show 
(see the discussion preceding
formula~(\ref{primelevelShah}) at the end of Section~\ref{section:extract}):
\begin{eqnarray} \label{primelevelShah2}
\Shan = 
\lm \frac{H^+}{H[I_e]^+ + K^+} \rmid \cdot 
\lm \frac{H[I_e]^+}{\Im  e + H[I_e]^+ \cap K^+} \rmid.
\end{eqnarray}

Our long term goal is to show that $\Shan$ divides~$\Shaf$. 
We have the following partial result towards this:

\begin{thm} \label{thm:main2}
Recall that we are assuming that $N$ is prime.   \\
\noindent (i) An odd prime~$q$ 
divides~$\Mid \frac{H^+}{H[I_e]^+ + K^+} \miD$
if and only if
there is a normalized eigenform~$g \in S_2(\Gamma_0(N), \C)$ such that
$L(g,1) = 0$ and $f$ is congruent to~$g$ 
modulo a prime ideal lying over~$q$ in the ring of integers
of the number field generated by the Fourier coefficients
of~$f$ and~$g$. \\
\noindent (ii) 
Assume that the first part of the Birch and Swinnerton-Dyer conjecture 
holds for all newform quotients of~$J_0(N)$ of positive analytic rank.
Let $q$ be a prime such that $q \ndiv N(N-1)$.
If $q$ divides the factor $\Mid \frac{H^+}{H[I_e]^+ + K^+} \miD$ 
of~$\Shan$,
then $q^2$ divides $\Mid \Sha(A_f) \miD$.
\end{thm}

Part~(i) of the theorem follows from Corollary~\ref{cor:divs} below.
Note that this part shows that the factor~$\Mid \frac{H^+}{H[I_e]^+ + K^+} \miD$
is non-trivial in general (cf. Remark~\ref{rmk:nontriv}).  
Part~(ii) follows from part~(i) using the theory of visibility
(see Theorem~\ref{thm:mainprime} below).
Thus one may say that  we relate a piece of~$\Shan$ to~$\Shaf$
using the theory of visibility.
Apart from 
providing additional theoretical evidence
towards the second part of the Birch and Swinnerton-Dyer conjecture, 
part~(ii) of the theorem also shows some consistency between the 
the implications of the two parts of the Birch and Swinnerton-Dyer conjecture.
Also, we remark that part~(i) of the theorem 
proves Theorem~\ref{thm:first} above, in view of the comments
preceding Theorem~\ref{thm:first} and the 
fact that if $A_g$ has positive Mordell-Weil rank,
then it has positive analytic rank by~\cite{kollog:finiteness}
(note that the proof of
Theorem~\ref{thm:main2} does not use Theorem~\ref{thm:first}).

We also have results similar to the ones above
when $N$ is not assumed to be prime, 
as we will explain soon. Before doing that, we pause to comment on
how our result fits into a potential approach towards proving the second
part of the Birch and Swinnerton-Dyer conjecture. The Euler system method gives results of
the type that $\Shaf$ divides~$\Shan$ (at present under certain 
hypotheses, and staying away from certain primes). In view of this,
what one would like to show the reverse divisibility, i.e., that
$\Shan$ divides~$\Shaf$. The Eisenstein ideal method of Skinner-Urban
works in this direction, as does the theory of visibility.
Our Theorem~\ref{thm:main2} may be seen as a first step in the
approach using visibility to try to show that $\Shan$ divides~$\Shaf$.
We hope to show in future work (using ideas similar 
to those in~\cite{ag:congnum}) that the entire factor
$\Mid \frac{H^+}{H[I_e]^+ + K^+} \miD$  of~$\Shan$
divides~$\Mid \Sha(A_f) \miD$, perhaps staying away from certain primes,
and under certain hypotheses, including the first part of the Birch
and Swinnerton-Dyer conjecture.
One major hurdle that remains is to
relate the other factor~$\Mid \frac{H[I_e]^+}{\Im  e + H[I_e]^+ \cap K^+} \miD$
in~(\ref{primelevelShah2}) to~$\Shaf$.
It is our hope that this factor can be explained by ``visibility
at higher level'' (see Remark~\ref{rmk:egs}(3) below).

We now drop the assumption
that $N$ is prime, and describe more general results as well as
discuss the organization of the paper, which in turn will also indicate the
steps in the proofs of our main results. 
In Section~\ref{section:formula}, we give a formula that expresses the left hand 
side~$\LAf(1)/\OAf$
of the Birch and Swinnerton-Dyer conjectural formula~(\ref{bsdform})
as a rational number. Using this formula,
in Section~\ref{section:extract}, we show  (Theorem~\ref{thm:mainform} below) that
\begin{eqnarray} \label{eqn:loo}
\frac{\LAf(1)}{\OAf} =
\frac{\Mid \frac{H^+}{H[I_e]^+ + K^+} \miD \cdot 
\Mid \frac{H[I_e]^+}{\Im  e + H[I_e]^+ \cap K^+} \miD }
{\Mid \pi_*(\T  e) / \pi_*(\Im  e) \miD},
\end{eqnarray}
up to powers of~$2$ and powers of primes whose
squares divide~$N$. 
Hence the Birch and Swinnerton-Dyer conjectural formula~(\ref{bsdform}) 
becomes: 
\begin{eqnarray} \label{eqn:int}
\frac{\Mid \frac{H^+}{H[I_e]^+ + K^+} \miD \cdot 
\Mid \frac{H[I_e]^+}{\Im  e + H[I_e]^+ \cap K^+} \miD }
{\Mid \pi_*(\T  e) / \pi_*(\Im  e) \miD}
\stackrel{?}{=}
\frac {\Mid \Sha(A_f) \miD \cdot \prod_{\scriptscriptstyle{p |N}}  c_p(A_f)}
      { \Mid A_f(\Q) \miD \cdot \Mid A_f^{\vee}(\Q) \miD },
\end{eqnarray}
up to powers of~$2$ and powers of primes whose
squares divide~$N$.
By Lemma~\ref{lem:divs} below, 
${\Mid \pi_*(\T  e) / \pi_*(\Im  e) \miD}$ divides
$\Mid A_f(\Q) \miD$.
In particular, the second part of the Birch and Swinnerton-Dyer conjecture 
implies  that the first factor
$\Mid \frac{H^+}{H[I_e]^+ + K^+} \miD$ 
in the numerator of the right side of
formula~(\ref{eqn:loo}), 
divides ${\Mid \Sha(A_f) \miD \cdot \prod_{\scriptscriptstyle{p |N}}  c_p(A_f)}$,
up to powers of~$2$ and powers of primes whose
squares divide~$N$. Our goal is to prove results towards this implication.
\comment{
In fact, if $N$ is prime, then things simplify significantly,
and the second part of the Birch and Swinnerton-Dyer conjecture just becomes
(see formula~(\ref{primelevelShah}) at the end of Section~\ref{section:extract}):
\begin{eqnarray} \label{primelevelShah2}
\Shan \stackrel{?}{=} 
\lm \frac{H^+}{H[I_e]^+ + K^+} \rmid \cdot 
\lm \frac{H[I_e]^+}{\Im  e + H[I_e]^+ \cap K^+} \rmid.
\end{eqnarray}
}

In Section~\ref{section:fac_to_congr}, we relate
the factor~$\Mid \frac{H^+}{H[I_e]^+ + K^+} \miD$ to the order of 
the intersection of certain abelian variety subquotients of~$J_0(N)$.
Then in Section~\ref{sec:inttocong}, we relate the order of
this intersection
to certain congruences of~$f$ with eigenforms of positive
analytic rank. In particular, we show that 
if $q$ is  
an odd prime such that $q^2 \ndiv N$, then $q$
divides~$\Mid \frac{H^+}{H[I_e]^+ + K^+} \miD$
if and only if
there is a normalized eigenform~$g \in S_2(\Gamma_0(N), \C)$ such that
$L(g,1) = 0$ and $f$ is congruent to~$g$ 
modulo a prime ideal lying over~$q$ in the ring of integers
of the number field generated by the Fourier coefficients
of~$f$ and~$g$ (see Corollary~\ref{cor:divs} below).
Thus we obtain the following byproduct
(for details, see Proposition~\ref{prop:conr_divs_Lvalue} and its proof):
\begin{prop} \label{prop:congphil}
Let $q$ be an odd prime such that $q^2 \ndiv N$ and 
$q$ does not divide $\Mid A_f(\Q)_{\rm tor} \miD$.
Suppose that there is a 
normalized eigenform~$g \in S_2(\Gamma_0(N), \C)$ such that 
$L(g,1) = 0$ and $g$ 
is congruent to~$f$
modulo a prime ideal over~$q$
in the ring of integers
of the number field generated by the Fourier coefficients
of~$f$ and~$g$. 
Then $q$ divides $\frac{\LAf(1)}{\OAf}$, and in particular,
$\frac{\LAf(1)}{\OAf} \equiv 
\frac{{{L_{\scriptscriptstyle{A_g}}}}(1)}
{{\Omega_{\scriptscriptstyle{A_g}}}} \bmod q$.
\end{prop}
A result similar to the one above
is proved in~\cite{dsw} using ideas from~\cite{vatsal:canonical}
(our proof is different).
These results fall under the general philosophy that congruences
between eigenforms should lead to congruences between algebraic
parts the special $L$-values (e.g., see~\cite{vatsal:canonical},
and the references therein for more instances). 

In Section~\ref{section:visibility}, we use the theory
of visibility to relate congruences  as in Proposition~\ref{prop:congphil} to
the product of the order of the Shafarevich-Tate group and the orders
of the component groups of~$A_f$, as we now explain.
Suppose $q$ and~$g$ are as in Proposition~\ref{prop:congphil}.
Assume the first part of the 
Birch and Swinnerton-Dyer conjecture 
for all newform quotients;
then the rank of~$A_g(\Q)$ is positive 
(considering that $L(g,1) =0$).
In this situation, a ``visibility theorem'' of~\cite{dsw} shows 
that under   
certain other technical hypotheses,
the congruence prime~$q$ 
divides 
$\Mid \Sha(A_f) \miD \cdot \prod_{p\mid N} c_p(A_f)$.
More precisely,  we obtain (see Theorem~\ref{thm:mainsqfree} for details):

\begin{thm} \label{thm:sqfree2}
Let $q$ be a prime such that 
$q$ divides $\Mid \frac{H^+}{H[I_e]^+ + K^+} \miD$. \\
Suppose that
$q \ndiv 2 N$, and for all maximal ideals~$\qq$ of~$\T$
with residue characteristic~$q$, $A_f[\qq]$ is irreducible.
Assume that for all newforms~$g$ of level dividing~$N$,
if $L(g,1) = 0$, then the rank of~$A_g(\Q)$ is positive
(this would hold if the first part of the Birch and Swinnerton-Dyer conjecture
is true).
Suppose that  
for all primes $p \divs N$, $p \not\equiv - w_p \pmod q$, where $w_p$
is the sign of the Atkin-Lehner involution acting on~$f$, 
and $p \not\equiv -1  \pmod q$ if $p^2 \divs N$. 
Suppose either that $f$
is not congruent modulo a prime ideal over~$q$ to a newform
of lower level (for Fourier coefficients of index coprime to~$Nq$),
or that there is a prime $p$ dividing~$N$ and
a maximal ideal~$\qq$ of~$\T$ with residue characteristic~$q$ 
such that
$f$ is congruent modulo~$\qq$ 
to a newform~$h$ of level dividing $N/p$ 
(for Fourier coefficients of index coprime to~$Nq$), with
$p^2 \ndiv N$, $w_p=-1$, and
$A_h[\qq]$ irreducible. \\
Then $q$ divides $\Mid \Sha(A_f) \miD \cdot 
\prod_{\scriptscriptstyle{p\mid N}} c_p(A_f)$.
\end{thm}

Thus, under certain hypotheses,
we show that if a prime 
divides~$\Mid \frac{H^+}{H[I_e]^+ + K^+} \miD$, then it
divides 
$\Mid \Sha(A_f) \miD \cdot \prod_{p\mid N} c_p(A_f)$, 
which, as mentioned earlier, is as implied by the second part of
the Birch and Swinnerton-Dyer conjecture. 
Also,
in some sense, one may say that our approach uses visibility to link the first
part of the Birch and Swinnerton-Dyer conjecture to the second.

Note that 
if $N$ is not prime, then the arithmetic component groups
intervene in trying to explain the Shafarevich-Tate group using 
the theory of visibility.
In fact, in Example~\ref{eg:cp} below, 
a prime divides
the factor $\Mid \frac{H^+}{H[I_e]^+ + K^+} \miD$,
but does not divide the Birch and Swinnerton-Dyer
conjectural order of~$\Sha(A_f)$,
and instead divides $\prod_{\scriptscriptstyle{p |N}}  c_p(A_f)$.
As opposed to some of the other approaches to the second part 
of the Birch and Swinnerton-Dyer conjecture, our approach
gives information about $\prod_{p\mid N} c_p(A_f)$ vis-a-vis
the conjecture. In fact, our result
seems to indicate that instead of considering 
the quantities $\Mid \Sha(A_f) \miD$ and~$\prod_{p\mid N} c_p(A_f)$ 
separately, one should consider the product
$\Mid \Sha(A_f) \miD \cdot \prod_{p\mid N} c_p(A_f)$ 
in approaches towards the second part of the Birch and Swinnerton-Dyer conjecture 
(at least in the approaches that use the theory of visibility).

We would also like to mention some of our speculations on
how to understand the second part of the Birch and Swinnerton-Dyer conjecture 
based on our formula~(\ref{eqn:loo}) for $\LAf(1)/\OAf$ even
when the level~$N$ is not prime.
Let $C_f$ denote the subgroup of~$A_f(\Q)$
generated by the image of~$\pi((0)-(\infty))$.
By Lemma~\ref{lem:divs}, $\pi_*(\T  e) / \pi_*(\Im  e) = C_f$,
and so ${\Mid \pi_*(\T  e) / \pi_*(\Im  e) \miD} = |C_f|$.
For simplicity, assume that $N$ is square free and that $A_f$
is an elliptic curve. Then, by~(\ref{eqn:int})
and the discussion above,
away from the prime~$2$, 
the second part of the Birch and Swinnerton-Dyer conjecture becomes:

\begin{eqnarray} \label{eqn:ec}
& &{\Big\Mid \frac{H^+}{H[I_e]^+ + K^+} \Big\miD \cdot 
\Big\Mid \frac{H[I_e]^+}{\Im  e + H[I_e]^+ \cap K^+} \Big\miD } \cdot
\Big\Mid \frac{A_f(\Q)}{C_f} \Big\miD
\cdot { \Mid A_f(\Q) \miD} \nonumber \\
& \stackrel{?}{=} &
{\Mid \Sha(A_f) \miD \cdot \prod_{\scriptscriptstyle{p |N}}  c_p(A_f)}\ \ .
\end{eqnarray}
One would like to understand how the various quantities on the left
are related to the quantities on the right. 
Based on some numerical data and theoretical results, we suspect that 
the product 
$\big( \Mid (A_f(\Q)/C_f) \miD
\cdot { \Mid A_f(\Q) \miD} \big)$ divides 
$\prod_{\scriptscriptstyle{p |N}}  c_p(A_f)$ (cf.~\cite{ag:lett}).
For example, when $N$ is prime, 
then by~\cite{emerton:optimal},
$$\Mid C_f \miD =
c_{\scriptscriptstyle N}(A_f) 
= \Mid A_f(\Q) \miD,$$
and so the two quantities 
$\big( \Mid (A_f(\Q)/C_f) \miD
\cdot { \Mid A_f(\Q) \miD} \big)$ and
$\prod_{\scriptscriptstyle{p |N}}  c_p(A_f)$
are actually equal.
If $N$ is not prime, then we do not expect equality 
(e.g., see Example~\ref{eg:cp} below),
but we expect that the ratio of
$\prod_{\scriptscriptstyle{p |N}}  c_p(A_f)$ to
$\big( \Mid (A_f(\Q)/C_f) \miD
\cdot { \Mid A_f(\Q) \miD} \big)$ is explained by congruences of~$f$
with
newforms of lower level having positive analytic rank 
(cf. Proposition~\ref{prop:cp} and Remark~\ref{rmk:cp} below).
Thus we suspect that the product
$ \Mid (A_f(\Q)/C_f) \miD
\cdot { \Mid A_f(\Q) \miD} $ on the left side of~(\ref{eqn:ec})
contributes to part of 
$\prod_{\scriptscriptstyle{p |N}}  c_p(A_f)$ in
the product $\Mid \Sha(A_f) \miD \cdot \prod_{p\mid N} c_p(A_f)$
on the right side of~(\ref{eqn:ec}),
and the rest of the product 
$\Mid \Sha(A_f) \miD \cdot \prod_{p\mid N} c_p(A_f)$ 
is explained by the product
${\Mid \frac{H^+}{H[I_e]^+ + K^+} \miD \cdot 
\Mid \frac{H[I_e]^+}{\Im  e + H[I_e]^+ \cap K^+} \miD }$
on the left side of~(\ref{eqn:ec}), via 
congruences of~$f$ with eigenforms of positive analytic rank,
possibly of some higher level (cf. Remark~\ref{rmk:egs}(3) below). 

While the picture above is largely speculative at this stage
and may need some extra conditions for it to be true (e.g., one
may have to stay away from certain primes), there 
are several reasons for mentioning it. Firstly, our Theorem~\ref{thm:sqfree2} 
may be seen as a first step in proving results towards our speculations above,
and shows that the theorem
fits into a bigger picture (albeit speculative).
Secondly, as opposed to other approaches to the second part of the Birch and Swinnerton-Dyer conjecture, our approach
predicts more explicitly how non-trivial elements in the various
quantities on the right side
of the second part of the Birch and Swinnerton-Dyer conjectural formula~(\ref{bsdform})
``arise''. We hope that this may lead to a better
understanding of the conjecture, even if it were proven by some
other means. Lastly, we hope that the broader speculative picture
above may help motivate computations or theoretical
results regarding parts of it.

We end our paper with an appendix  in Section~\ref{sec:app}, 
where we prove that for a prime~$p$, 
 the prime-to-$p$ parts of the 
component group at~$p$ used in the Birch and Swinnerton-Dyer conjecture
and the component group at~$p$
used in the formulation of the Bloch-Kato conjecture
are equal. While the result
is well known, to our knowledge, it has not been proven as such
in the literature.

We remark that throughout this article, the notation
introduced in one section is carried over to the next. \\

\noindent {\it Acknowledgements:} This paper owes its existence to 
Lo\"ic Merel. It was his idea to extract a factor as above and 
he expected that one can relate the factor
to the order of the Shafarevich-Tate group using 
Mazur's theory of visibility. The author's task was to work out
the details and see how far this idea can be taken.
We also thank Neil Dummigan for several
helpful conversations regarding 
Section~\ref{section:visibility}, and Dipendra Prasad and Minhyong 
Kim for discussions regarding the Appendix. 
We would also like to thank the anonymous referee for 
suggestions that improved the presentation of this paper.
Part of the work was done during visits to 
the Institut des Hautes \'Etudes Scientifiques and
the Tata Institute for Fundamental Research; we are grateful
to both institutions for their kind hospitality.

\section{A formula for the ratio of special $L$-value
to real volume for newform quotients}
\label{section:formula}

The goal of this section is to give a formula (formula~(\ref{formula:lovero}) 
below) 
that expresses~$\LAf(1)/\OAf$, the left hand side of~(\ref{bsdform}), 
as a rational number. 

Let $\mathcal{H}$ denote the complex upper half plane, and
let ${\{0,i\infty\}}$ denote the projection of the geodesic
path from~$0$ to~$i\infty$ in~${\mathcal{H}} \cup {\PP}^1 ({\Q})$
to~$X_0(N)({\C})$.
We have an isomorphism
$$H_1(X_0(N),{\Z}) \tensor {\R} \stackrel{\cong}{\lra}
{\rm Hom}_{{\C}} (H^0(X_0(N), \Omega^1),{\C}),$$ obtained
by integrating differentials along cycles (see~\cite[\S~IV.1]{lang:modular}).
Let $e$ be the element of~$H_1(X_0(N),{\Z}) \tensor {\R}$ that corresponds
to the map $\omega \mapsto - \int_{\{0,i\infty\}} \omega$ under this
isomorphism. It is called the {\em winding element}.

By the Manin-Drinfeld Theorem
(see~\cite[Chap. {I}{V}, Theorem~$2.1$]{lang:modular} 
and~\cite{manin:parabolic}),
\mbox{$e \in H_1(X_0(N),\Z) \tensor \Q$}. Also, since the complex
conjugation involution on $H_1(X_0(N),\Z)$ is induced by the map
$z \mapsto - \overline{z}$ on the complex upper half plane,
we see that $e$ is invariant under complex conjugation.
Thus $\T e \subseteq H_1(X_0(N),\Z)^+ \tensor \Q$.
If the torsion-free group $\pi_*(\T e)$ has rank equal
to~$d$ ($= \dim A_f$), then let 
$[ H_1(A_f,{\Z})^+ : \pi_*(\T e) ]$ denote
the absolute value of the
determinant of an automorphism of~$H_1(A_f,\Q)$ that
takes the lattice $H_1(A_f,{\Z})^+$ isomorphically onto
the lattice $\pi_*(\T e)$; otherwise, define
$[ H_1(A_f,{\Z})^+ : \pi_*(\T e) ]$ to be zero.

Let $g_1, \ldots , g_d$ 
be a $\Z$-basis of $S_2(\Gamma_0(N),\Z)[I_f]$, and for $j = 1, \ldots, d$,
consider $\omega_j = 2 \pi i g_j(z) dz$ as differentials
in~$H^0(A_f,\Omega_{A_f/{\Q}})$.
There exists $c \in \Q^*$ such that 
$D = c \cdot \wedge_j \omega_j$.
As in~\cite{agst:manin}, we call the absolute value of~$c$ 
the Manin constant of~$A_f$,
and denote it by $\caf$.
Let $c_\infty(A)$ denote the number of connected components of~$A({\R})$.

The following theorem is similar to Theorem~4.5 in~\cite{agst:bsd}
and the key idea behind the proof goes back to~\cite{agashe:invis}.

\begin{thm} \label{thm:lovero}
\begin{eqnarray} \label{formula:lovero}
\frac{\LAf(1)}{\OAf} = 
\frac{\left[ H_1(A_f,{\Z})^+ : \pi_*(\T e) \right]}
     {\caf \cdot c_\infty(A_f)}. 
\end{eqnarray}
\end{thm}

The rest of this section is devoted to proving the theorem above,
and apart from Lemma~\ref{perfpair4}, none of the discussion
in the rest of this section will be used later.

We start by giving some lemmas that will be used in the proof
of the theorem.
Let $S_f = S_2(\Gamma_0(N),\Z)[I_f]$\label{sym:sf}. 
There is a perfect pairing 
\begin{eqnarray}\label{perfpair3}
\T \times S_2(\Gamma_0(N),\Z) \ra \Z
\end{eqnarray}
which associates to~$(T,f)$ the first Fourier coefficient~$a_1(f \divs T)$ of 
the modular form~$f \divs T$ (see \cite[(2.2)]{ribet:modp});
this induces a pairing
$$\psi: \T/I_f \times S_f \ra \Z.$$
 
\begin{lem}\label{perfpair4}
The pairing~$\psi$ above is a perfect pairing.
\end{lem}
\begin{proof}
Both $\T/I_f$ and $S_f$ are free \mbox{$\Z$-modules}
of the same rank.
So it suffices to prove that the induced
maps $S_f \ra \Hom(\T/I_f, \Z)$ and $\T/I_f \ra \Hom(S_f, \Z)$
are injective. The injectivity of the first map follows
from the perfectness of the pairing~(\ref{perfpair3}).
Suppose the image of $T \in \T$ in $\T/I_f$
maps to the trivial element of $\Hom(S_f, \Z)$.
Then $a_1(f \divs T) =0$. But $f$ is an eigenform for~$T$; suppose
the eigenvalue is~$\lambda$. Then $0 = a_1(f \divs T) = \lambda a_1(f)
= \lambda$. Thus $f \divs T = 0$, i.e., $T \in I_f$.
Thus the map
$\T/I_f \ra \Hom(S_f, \Z)$ is injective and we are done.
\end{proof}

\begin{lem} \label{te}
The map $\T \ra \T e$ given by $t \mapsto te$
induces an isomorphism $\T/I_f \stackrel{\cong}{\lra}\T e/ \If e$. 
\end{lem}

\begin{proof}
It is clear that the map $\T \ra \T e/ \If e$ given by $t \mapsto te$
is surjective. All we have to show is that
the kernel of this map is~$I_f$.
It is clear that the kernel contains~$I_f$. Conversely, 
if $t$ is in the kernel, then $te \in \If e$; let
$i \in I_f$ be such that $te = ie$. Then $(t-i)e=0$,
and thus $\int_{(t-i)e} \omega_f = 0$, i.e., 
$\int_e \omega_{(t-i)f} =0$. If the eigenvalue
of~$f$ under~$(t-i)$ is~$\lambda$, then this means
$\lambda \cdot L(f,1) = 0$, i.e., $\lambda = 0$. Thus
$(t-i) \in I_f$, i.e., $t \in I_f$.
\end{proof}

We have a pairing
\begin{eqnarray} \label{eqn:pairing}
H_1(X_0(N),\Z) \tensor {\C} \times  S_2(\Gamma_0(N), {\C}) \ra {\C}
\end{eqnarray}
given by~$(\gamma,f)\mapsto \langle \gamma,f\rangle  = \int_\gamma \omega_f$
and extended $\C$-linearly.
At various points below in this section, we will consider pairings 
between two $\Z$-modules; unless otherwise stated,
each such pairing is obtained in a natural
way from~(\ref{eqn:pairing}).
If $\langle\ ,\ \rangle : M \times M' \ra \C$, 
is a pairing between two $\Z$-modules~$M$ and~$M'$, each
of the same rank~$m$, and 
$\{\alpha_1, \ldots, \alpha_m \}$
and $\{ \beta_1, \ldots, \beta_m\}$ are 
bases of~$M$ and~$M'$ (respectively),
then by~${\rm disc}(M \times M' \ra \C)$, we mean
the absolute value of~$\det (\langle \alpha_i, \beta_j\rangle )$;
this value is independent of the choices of bases made in its definition.

\begin{lem}\label{lem:OAf}
$\OAf = \caf \cdot c_\infty(A_f) \cdot
{\rm disc}(H_1(A_f,\Z)^+ \times S_f \ra {\C})$.
\end{lem}

\begin{proof}
Let $A(\R)^0$ denote the component of~$A(\R)$ containing
the identity. Then
\begin{tabbing}
$\OAf$ \= $= \int_{A(\R)} D = c_\infty(A_f) \cdot \int_{A(\R)^0} D
= c_\infty(A_f) \cdot \caf \cdot \int_{A(\R)^0} \wedge_j \omega_j$ \\
\> $= c_\infty(A_f) \cdot \caf \cdot {\rm disc}(H_1(A(\R)^0,\Z) \times S_f \ra {\C})$
\\
\> $= c_\infty(A_f) \cdot \caf \cdot {\rm disc}(H_1(A_f,\Z)^+ \times S_f \ra {\C})$,
\end{tabbing}
where the last equality follows from the canonical isomorphism
$H_1(A(\R)^0,\Z) \isom H_1(A_f,\Z)^+$ (see, e.g., \cite[Lemma~4.4]{agst:bsd};
note that several of the $A(\R)$'s in Section~4.2 of~\cite{agst:bsd}
should really be $A(\R)^0$'s).
\end{proof}

\begin{proof}[Proof of Theorem~\ref{thm:lovero}]
Let $\{f_i\}$, for $i = 1, 2, \ldots d$, denote the set of
Galois conjugates of~$f$.
By~\cite[Thm.~7.14]{shimura:intro}
and~\cite{carayol:hilbert}, $L_{A_f}(s) = \prod_i L(f_i,s) 
= \prod_i \langle e,f_i\rangle$. 
Hence, using Lemma~\ref{lem:OAf},

\begin{eqnarray*}
\caf \cdot c_\infty(A_f) \cdot \frac{\LAf(1)}{\OAf}  
& = &\frac {\prod_i \langle e,f_i\rangle }{{\rm disc}(H_1(A_f,\Z)^+ \times S_f \ra {\C})} \\
& = &\frac {\prod_i \langle e,f_i\rangle }{{\rm disc}(\T e/\If e \times S_f \ra {\C})} \cdot 
\left[ H_1(A_f,{\Z})^+ : \pi_*({\T}e) \right].
\end{eqnarray*}

To prove the theorem, it suffices to prove that
\begin{eqnarray} \label{keyclaim}
\frac {\prod_i \langle e,f_i\rangle }
{{\rm disc}(\T e/\If e \times S_f \ra {\C})} = 1.
\end{eqnarray}

In what follows,
$i, j, k, {\rm and\ } \ell$ are indices running from~$1$ to~$d$.
Let $\{g_k\}$ be a~$\Z$-basis of~$S_f$ and let $\{t_j\}$ be the 
corresponding dual basis
of~${\T}/I_f$ under the perfect pairing~$\psi$ in Lemma~\ref{perfpair4} above. 
Then by Lemma~\ref{te}, $\{t_j e\}$ is a basis for~${\T} e / \If e$.
Now $g_k = \sum_k a_{ki} f_i$ for some $\{a_{ki} \in \C\}$. Let
$A$ be the matrix having $(k,i)$-th entry~$a_{ki}$,
and let $(a^{-1})_{i\ell}$
denote the $(i,\ell)$-th element of the inverse of~$A$. Then 

\begin{tabbing}
\= ${\rm disc} ({\T} e / \If e \times S_f \ra {\C})$ \\
\> $ = \det \{\langle t_j e,g_k\rangle \}= \det \{\langle e,g_k \divs t_j\rangle \} 
   = \det \{\langle e,(\sum_i a_{ki} f_i)\divs t_j\rangle \}$\\ 
\> $= \det \{\langle e,\sum_i a_{ki} a_1(f_i \divs t_j) f_i\rangle \}$ 
   \ \ \ \ (since $f_i$'s are eigenvectors) \\
\> $= \det \{\langle e,\sum_i a_{ki} \sum_\ell (a^{-1})_{i\ell} a_1(g_\ell \divs t_j) f_i\rangle \}$
   \ \ (using $f_i= \sum_\ell (a^{-1})_{i\ell} g_\ell$)\\
\> $= \det \{\langle e,\sum_i a_{ki} (a^{-1})_{ij} f_i\rangle \}$
   \ \ \ \ (using $a_1(g_\ell \divs t_j) = \delta_{\ell j}$)\\
\> $= \det \{\sum_i a_{ki} (a^{-1})_{ij} \langle e,f_i\rangle \} 
= \det \{\sum_i a_{ki} \langle e,f_i\rangle  (a^{-1})_{ij} \}$\\ 
\> $= \det(A \Delta A^{-1})$
  \ \ \ \  (where $\Delta = {\rm diag}(\langle e,f_i\rangle )$)\\
\> $= \det (\Delta) = \prod_i \langle e,f_i\rangle .$
\end{tabbing}

This proves~(\ref{keyclaim}), and finishes the proof of the theorem.
\end{proof}

\section{Extracting an integer factor}
\label{section:extract}

The goal of this section is to rewrite the right 
hand side of formula~(\ref{formula:lovero}) so that
it can be better compared to the formula given by
the Birch and Swinnerton-Dyer conjecture
(Conjecture~\ref{bsd2}). We also extract the
factor $\Mid \frac{H^+}{H[I_e]^+ + K^+} \miD$
from~$\frac{\LAf(1)}{\OAf}$. 
Recall that $H=H_1(J_0(N),\Z)$, $I_e$ is the annihilator ideal of~$e$,
$K$ denotes the kernel of~$\pi_*$ restricted to~$H$,
and $\Im$ denotes the annihilator ideal
of the divisor $(0) - (\infty)$.
\begin{lem} \label{lem:Ie}
$\Im e \subseteq H[I_e]^+$.
\end{lem}
\begin{proof}
By~\cite[II.18.6]{mazur:eisenstein}, we have $\Im e \subseteq H^+$
(note that in loc. cit., the definition of $\Im$ is different and $N$ 
is assumed to be prime; but the only essential property of $\Im$ that
is used in the proof is that $\Im$ annihilates the divisor $(0) - (\infty)$,
and the assumption that $N$ is prime is not used). Also, $e$ 
is killed by~$I_e$, hence so is~$\Im e$, and the lemma follows.
\end{proof}

\begin{thm} \label{thm:mainform}
Up to a power of~$2$,
\begin{eqnarray}\label{mainform}
\caf \cdot c_\infty(A_f)\cdot \frac{\LAf(1)}{\OAf} =
\frac{\Mid \frac{H^+}{H[I_e]^+ + K^+} \miD \cdot 
\Mid \frac{H[I_e]^+}{\Im  e + H[I_e]^+ \cap K^+} \miD }
{\Mid \pi_*(\T  e) / \pi_*(\Im  e) \miD}.
\end{eqnarray}
\end{thm}

\begin{proof}
Since $\Im e \subseteq H^+$, 
and the map $\pi_*: H \ra H_1(A_f, \Z)$
is surjective (as the kernel of $J \ra A_f$ is connected),
from Theorem~\ref{thm:lovero}, we have
\begin{eqnarray} \label{inproof}
\caf \cdot c_\infty(A_f)\cdot \frac{\LAf(1)}{\OAf}  
= [\pi_*(H)^+ : \pi_*(\T  e)]  
= \frac{\Mid \pi_*(H)^+ /\pi_*(\Im  e) \miD}
{\Mid \pi_*(\T  e) / \pi_*(\Im  e) \miD}.
\end{eqnarray}

Now  $\pi_*(H) = H/K$ and up to a power of~$2$,
$(H/K)^+ = H^+/K^+$.
Also, $\pi_*(\Im e) = \Im e/(\Im e \cap K^+)$
and the kernel of the natural injective map $\Im e \ra ( \Im e + K^+ )/K^+$
is $\Im e \cap K^+$; so $\pi_*(\Im e) =  ( \Im e + K^+ )/K^+$.
Thus up to a power of~$2$,
\begin{eqnarray} \label{eqn:oddeq}
\Mid \pi_*(H)^+ /\pi_*(\Im  e) \miD =
\lm \frac{H^+}{ K^+ + \Im  e} \rmid.
\end{eqnarray}

Now $\Im e \subseteq H[I_e]^+$, and so we have
$$\lm \frac{H^+}{K^+ + \Im  e} \rmid
= \lm \frac{H^+}{H[I_e]^+ + K^+} \rmid \cdot 
\lm \frac{H[I_e]^+  + K^+}{\Im  e + K^+} \rmid.$$

We have a map $H[I_e]^+ \ra (H[I_e]^+  + K^+)/ (\Im  e + K^+)$
given by $h \mapsto h + (\Im  e + K^+)$. It is clearly
a surjection, and its kernel is precisely $\Im  e + H[I_e]^+ \cap K^+$.
So we have 
$$\lm \frac{H^+}{K^+ + \Im  e} \rmid
= \lm \frac{H^+}{H[I_e]^+ + K^+} \rmid \cdot 
\lm \frac{H[I_e]^+}{\Im  e + H[I_e]^+ \cap K^+} \rmid.$$
Putting this in~(\ref{eqn:oddeq}) and using~(\ref{inproof}), 
we get the theorem.
\end{proof}

By~\cite{agst:manin}, if a prime $p$ divides~$\caf$, then $p=2$ or
$p^2 \divs N$. Also, $c_\infty(A_f)$ is a power of~$2$.
Hence, in view of Theorem~\ref{thm:mainform}, the second part of
the Birch and Swinnerton-Dyer conjecture (Conjecture~\ref{bsd2})
just says that up to powers of~$2$ and powers of primes whose
squares divide~$N$, we have
\begin{eqnarray} \label{sqfreelevelShah}
\frac{\Mid \frac{H^+}{H[I_e]^+ + K^+} \miD \cdot 
\Mid \frac{H[I_e]^+}{\Im  e + H[I_e]^+ \cap K^+} \miD }
{\Mid \pi_*(\T  e) / \pi_*(\Im  e) \miD} \stackrel{?}{=}
\frac {\Mid \Sha(A_f) \miD \cdot \prod_{\scriptscriptstyle{p |N}}  c_p(A_f)}
      { \Mid A_f(\Q) \miD \cdot \Mid A_f^{\vee}(\Q) \miD }.
\end{eqnarray}

\begin{lem}[A-Stein] \label{lem:divs}
$\pi_*(\T  e) / \pi_*(\Im  e)$ is the subgroup of~$A_f(\Q)$
generated by the image of~$\pi((0)-(\infty))$ and so
${\Mid \pi_*(\T  e) / \pi_*(\Im  e) \miD}$ divides
$\Mid A_f(\Q) \miD$.
\end{lem}
\begin{proof}
The first claim follows 
from the proof of Prop.~4.6 of~\cite{agst:bsd}
and the second claim follows from the first
by Lagrange's theorem that the order of a subgroup divides the
order of the group.
\end{proof}

Thus the second part of
the Birch and Swinnerton-Dyer conjecture 
implies that up to powers of~$2$ 
and powers of primes whose square divides~$N$,
the factor $\Mid \frac{H^+}{H[I_e]^+ + K^+} \miD$
divides $\Mid \Sha(A_f) \miD \cdot \prod_{p\mid N} c_p(A_f)$.

If $N$ is prime, then things simplify significantly,
since $\caf$ is a power of~$2$ by~\cite{agst:manin}
and by~\cite{emerton:optimal},
we have ${\Mid \pi_*(\T  e) / \pi_*(\Im  e) \miD} = 
c_{\scriptscriptstyle N}(A_f) 
= \Mid A_f(\Q) \miD = \Mid A_f^{\vee}(\Q) \miD$.
Hence if $N$ is prime, the Birch and Swinnerton-Dyer conjecture says 
that up to powers of~$2$,
\begin{eqnarray} \label{primelevelShah}
\lm \frac{H^+}{H[I_e]^+ + K^+} \rmid \cdot 
\lm \frac{H[I_e]^+}{\Im  e + H[I_e]^+ \cap K^+} \rmid \stackrel{?}{=} 
\Mid \Sha(A_f) \miD.
\end{eqnarray}

In particular, we note the following result for easy reference:

\begin{prop}\label{prop:divs}
If $N$ is prime, then 
the second part of the Birch and Swinnerton-Dyer 
conjecture implies that the odd part of 
$\Mid \frac{H^+}{H[I_e]^+ + K^+} \miD$ divides $\Mid \Sha(A_f) \miD$.
\end{prop}

\section{Relating the factor to an intersection}
\label{section:fac_to_congr}

In this section, we will relate the factor 
$\Mid \frac{H^+}{H[I_e]^+ + K^+} \miD$ to the order of
the intersection of certain subquotients of~$J$.

We first need the following lemma, which
is an adaptation of~\cite[Prop.~3.2]{agst:bsd}.
Suppose $J$ is an abelian variety over~$\Q$,
and $A$ and $B$ are abelian subvarieties of~$J$.
Let $(A \cap B)^0$ denote the connected component 
of~$A \cap B$ that contains the identity; it is
an abelian variety over~$\Q$. Let $(A\cap B)^c$
denote the component group of~$A\cap B$, i.e.,
$(A\cap B)/(A \cap B)^0$.
If $G$ is a finitely generated abelian group, then we denote
its torsion part by~$G_{\tor}$.

\begin{lem}\label{lemma:intersection}
With notation as above, there is a natural isomorphism of groups
$$(A\cap B)^c \isom 
  \left(\frac{H_1(J,\Z)}{H_1(A,\Z) + H_1(B,\Z)}\right)_{\tor.}$$
\end{lem}
\begin{proof}
If $E$ is an abelian variety over~$\Q$, then let
$T_E$ denote the vector space $\Hom_{\C}(H^0(E_\C, \Omega^1), \C)$,
and let $H_E$ denote $H_1(E(\C), \Z)$.
There is a natural inclusion $H_E \ra T_E$, and 
the complex torus $E(\C)$ is isomorphic
to $T_E/H_E$. 
There is also a natural inclusion of~$T_A$ and~$T_B$ in~$T_J$,
and we have $H_A = H_J \cap T_A$ and $H_B = H_J \cap T_B$.

The exact sequence
$$0\ra A\cap B \ra A \oplus B \ra J$$
fits into the following diagram obtained from complex analytic uniformization.
$$\xymatrix{
& 0 \ar[d]  &  0 \ar[d]& \\
H_{(A \cap B)^0} \ar[d] \ar[r] 
& {H_A \oplus H_B}\ar[d] \ar[r] & {H_J} \ar[r]\ar[d]&
            {H_J/(H_A + H_B)}\ar[d] \\
T_{(A \cap B)^0} \ar[r]\ar[d] & {T_A \oplus T_B}\ar[d] \ar[r] & T_J \ar[r]\ar[d] & {T_J/(T_A+T_B)}\ar[d]\\
{A\cap B}\ar[r] & A\oplus B\ar[r] \ar[d]& J  \ar[d] \ar[r] & J/(A+ B)\\
& 0   &  0 & 
}$$
Using the snake lemma, which connects the kernel $A\cap B$ of 
$A\oplus B \ra J$ to the cokernel 
of $H_A \oplus H_B \ra H_J$, we obtain an exact sequence
$$0 \ra T_{(A \cap B)^0}/H_{(A \cap B)^0} \ra 
   A\cap B \ra 
   H_J/(H_A + H_B) \ra
   T_J/(T_A+T_B).$$
Now $T_{(A \cap B)^0}/H_{(A \cap B)^0} = (A\cap B)^0$,
so the exact sequence becomes
$$0 \ra   (A\cap B)^c \ra 
   H_J/(H_A + H_B) \ra
   T_J/(T_A+T_B).$$
Since $T_J/(T_A+T_B)$ is a $\C$-vector space, the torsion 
part of $H_J/(H_A + H_B)$ must map to~$0$.
No non-torsion in $H_J/(H_A + H_B)$ could
map to~$0$, because if it did then $(A\cap B)^c$ would not
be finite.  The lemma follows.
\end{proof}

In this paragraph and the next, the symbol~$g$ stands
for a newform of level~$N_g$ dividing~$N$. 
Let $S'_g$ denote the subspace of~$S_2(\Gamma_0(N),\C)$
spanned by the forms $g(dz)$ where $d$ ranges over the
divisors of~$N/N_g$. If $n$ is a positive integer, then
let $\sigma_0(n)$ denote the number of divisors of~$n$.
Then $S'_g$ has dimension~$\sigma_0(N/N_g)$.
Let $[g]$ denote the Galois orbit of~$g$, 
and let $S_{[g]}$ denote the $\Q$-subspace of
forms in~$\oplus_{h \in [g]} S'_h$ with rational Fourier
coefficients. We have 
$S_2(\Gamma_0(N),\Q) = \oplus_{[g]} S_{[g]}$, where
the sum is over Galois conjugacy classes of newforms of
some level dividing~$N$.

Let $B_g$ denote the abelian subvariety of~$J_0(N_g)$
associated to~$g$ by Shimura~\cite[Thm.~7.14]{shimura:intro}, 
and let $J_g$ be the sum of
the images of~$B_g$ in~$J=J_0(N)$
under the usual degeneracy maps.
Then $J_g$ is isogenous to $B_g^{\sigma_0(N/N_g)}$.
If $J'$ is an abelian subvariety of~$J$ that is preserved by $\End J$,
then by $\Ann_{\T \tensor \Q} J'$ we mean the kernel of the
image of $\T \tensor \Q$ in $\End J' \tensor \Q$.
Note that $\End J$ preserves~$J_g$ (e.g., see~\cite[\S~3]{ars:moddeg})
and $\Ann_{\T \tensor \Q} J_g = 
\Ann_{\T \tensor \Q} S_{[g]}$. 
If $T$ is a subset of the set of Galois conjugacy classes of newforms of
some level dividing~$N$, then let 
$S_T = \oplus_{[g] \in T} S_{[g]}$, $I_T = \Ann_\T S_T$, and let 
$J_T$ denote the abelian subvariety of~$J$ generated
by the~$J_g$ for all $[g] \in T$. Then $J_T$ is isogenous
to $\prod_{[g] \in T} J_g$, hence is preserved by~$\End J$ and 
moreover,
$\Ann_{\T \tensor \Q} J_T = 
\Ann_{\T \tensor \Q} S_T$. 
Intersecting with~$\T$, we get

\begin{lem} \label{lem:ann}
$\Ann_\T J_T = \Ann_\T S_T = I_T$.
\end{lem}

\comment{
Let $A_S = J / J_S$, so that we have a short exact sequence
$$0 \ra J_S \ra J \ra A_S \ra 0,$$
part of whose associated long exact sequence of homology is:
$$\ldots \ra H_0(A_S, \Z) \ra H_1(J_S,\Z) \ra H_1(J,\Z)
\ra H_1(A_S,\Z) \ra \ldots.$$
Since $A_S$ is connected and $H_1(A_S,\Z)$ is torsion-free,
we have a natural injection of~$H_1(J_S, \Z)$ into~$H$, 
with torsion-free cokernel.
}

We say that a subgroup~$G'$ of a finitely generated abelian group~$G$
is {\em saturated} (in~$G$) if the quotient~$G/G'$ is torsion-free.
Recall that $H = H_1(J_0(N),\Z)$. 

\begin{lem} \label{lem:annh1}
$H_1(J_T, \Z) = H[I_T]$.
\end{lem}
\begin{proof}
We have $H_1(J_T, \Z) = H \cap \Hom_{\C}(H^0(J_T(\C), \Omega^1), \C)$.
Hence $H_1(J_T, \Z)$ is saturated in~$H$ and 
it is clear by Lemma~\ref{lem:ann}
that $H_1(J_T, \Z) \subseteq H[I_T]$. 
So it suffices to show that the inclusion
$H_1(J_T, \Z) \tensor \Q \subseteq H[I_T] \tensor \Q$ is
an equality, for which in turn, it suffices to check
that the two $\Q$-vector spaces have the same dimension.
Now $H \tensor \Q$ is free of rank~$2$ over~$\T \tensor \Q$,
and thus, $\dim_\Q H[I_T] = 2 \cdot \dim_\Q (\T/ I_T \tensor \Q)$. 
We have
\begin{tabbing}
\= $\dim_\Q (H_1(J_T, \Z) \tensor \Q)$ \\
\> $= 2 \cdot \dim J_T = 2 \cdot \dim_\Q H^0(J_T(\Q), \Omega_{J_T(\Q)/\Q})$ \\
\> $= 2 \cdot \sum_{[g] \in T} \dim_\Q H^0(J_g(\Q), \Omega_{J_g(\Q)/\Q}) 
= 2 \cdot \sum_{[g] \in T} \dim_\Q T_{[g]}$ \\
\> $= 2 \cdot \dim_\Q S_T = 2 \cdot \dim_\Q (\T/I_T \tensor \Q)$, 
\end{tabbing}
where the last equality follows 
since 
$S_2(\Gamma_0(N),\Q)$ is free of rank one
over~$\T\tensor \Q$.
Thus $\dim_\Q H_1(J_T, \Z) \tensor \Q = 
\dim_\Q H[I_T] \tensor \Q$, and we are done.
\end{proof}

Let $T_1$, $T_2$, and~$T_3$ be the sets of Galois conjugacy 
classes of newforms of level dividing~$N$ of analytic rank zero,
of classes not in~$[f]$, and of classes of newforms of analytic
rank zero except those in~$[f]$ respectively.
Let $A = J_{T_1}$, $B = J_{T_2}$, and $C = J_{T_3}$.
Also for simplicity, we will write $I_A$ for~$I_{T_1}$,
$I_B$ for~$I_{T_2}$, and $I_C$ for~$I_{T_3}$, as well as
$S_A$ for~$S_{T_1}$,
$S_B$ for~$S_{T_2}$, and $S_C$ for~$S_{T_3}$.
Note that both $I_A$ and $I_B$ are contained in~$I_C$,
and $C$ is an abelian subvariety of
both~$A$ and~$B$. The differentials
on $A/C$ and $B/C$ correspond
respectively to the space generated by
the conjugates of~$f$ and the space generated by eigenforms of analytic rank
greater than zero.

\begin{lem} \label{prop:intersect}
$${\rmid \frac{H}{H[I_e] + K} \rmid} =
\Mid (A/C) \cap (B/C) \miD\ ,$$
where the latter intersection is considered in $J/C$.
\end{lem}

\begin{proof}
{\em Claim~1:} $H_1(A,\Z)  = H[I_e]$.
\begin{proof}
By Lemma~\ref{lem:annh1},
it suffices to show that $I_A = I_e$.
It follows from 
Lemma~3.10 in~\cite{parent} that
$I_e \subseteq I_A$. We next prove the reverse inclusion.
Let $t \in \T$.
By the first statement in the proof
of Lemma~3.10 in~\cite{parent}, 
if $f$ is a newform with analytic rank
greater than zero of level~$M$ dividing~$N$, and
$g \in S'_f$, then $\langle e, g \rangle = 0$; 
hence $\langle t e, f \rangle = \langle e, tf \rangle = 0$.
If moreover $t \in I_A$, then for all newforms~$g$ of analytic
rank zero, 
$\langle t e, g \rangle = \langle e, tg \rangle = \langle e, 0 \rangle = 0$.
Thus $te$ is orthogonal to all of $S_2(\Gamma_0(N),\C)$,
and so $te = 0$, i.e., $t \in I_e$, which shows that $I_A \subseteq I_e$,
and we are done.
\end{proof}

Recall that $K$ denotes the kernel of~$\pi_*: H_1(J_0(N),\Q) \ra H_1(A_f,\Q)$
restricted to~$H$.\\

\noindent{\em Claim~2:} $H_1(B,\Z)  = K$.
\begin{proof}
Now $B$ is easily seen to be
the kernel of the map $J \ra A_f$, and so we have
a short exact sequence
$0 \ra B \ra J \ra A_f \ra 0$, and part of the associated
long exact sequence of homology is:
$$\ldots \ra H_1(B,\Z) \ra H_1(J,\Z)
\stackrel{\pi_*}{\ra} H_1(A_f,\Z) \ra 0 \ra \ldots$$
Hence it is clear that $H_1(B,\Z) \subseteq K$.
Since $H_1(B,\Z)$ is saturated in~$H$, it suffices
to show that $H_1(B,\Z) \tensor \Q = K \tensor \Q$,
i.e., that $H_1(B,\Z)$ and~$K$ have the same rank.
If $G$
is a free abelian group, then we denote its rank by~$\rank(G)$.
Now 
\begin{tabbing}
\= $\rank(K) = 2 \cdot \dim J - 2 \cdot \dim A_f
= 2 \cdot \dim_\Q S_2(\Gamma_0(N),\Q) - 2 \cdot \dim_\Q S_{[f]}$ \\ 
\> $ = 2 \cdot \dim_\Q S_{T_2} = 2 \cdot \dim_\Q B = \rank(H_1(B,\Z))$.
\end{tabbing} This proves the claim.
\end{proof}

Also, $(A \cap B)^0 = C$, and so
$(A \cap B)^c = \Mid (A/C) \cap (B/C) \miD$.
The desired result now follows 
by Lemma~\ref{lemma:intersection} in view of the two claims above.
\end{proof}

\section{Relating the intersection to certain congruences}
\label{sec:inttocong}

If $T$ is a subset of the set of Galois conjugacy classes of newforms of
some level dividing~$N$, then let $T'$ denote its complement.
For simplicity, we will write $\IA$ for~$I_{T'_1}$,
$\IB$ for~$I_{T'_2}$, and $\IC$ for~$I_{T'_3}$, as well as
$\SA$ for~$S_{T'_1}$,
$\SB$ for~$S_{T'_2}$, and $\SC$ for~$S_{T'_3}$.
Note that since $T'_2$ is the set of Galois conjugacy classes of 
newforms of level dividing~$N$ that are in~$[f]$, we have
$\IB = I_f$.
Let $S$ be short for $S_2(\Gamma_0(N),\Z)$.
Recall that the exponent of a finite group~$G$ is the smallest
positive integer that annihilates~$G$.
\begin{lem} \label{prop:inter_to_cong}
The exponent of the group
$ (A/C) \cap (B/C)$
divides the exponent of the group
$\frac{S[\IC]}{S[\IB] + S[\IA]}$.
Moreover, if $\ell$ is a prime such that $\ell^2 \nmid N$,
then $\ell$ does not divide the ratio of the exponent
of $\frac{S[\IC]}{S[\IB] + S[\IA]}$ to the exponent
of~$ (A/C) \cap (B/C)$.
\end{lem}

\begin{proof}
Both parts of the lemma are generalizations of results in~\cite{ars:moddeg}.

The proof of the first part is a generalization 
of the proof of
Theorem~3.6(a) in~\cite{ars:moddeg} which says that the modular exponent
divides the congruence exponent. 
Note that $\End J_0(N)$ preserves~$A$, $B$, and~$C$,
and hence~$A/C$ and~$B/C$,
and that the image
of~$\T$ acting on~$J_0(N)/C$ is~$\TCp$
(since $\Ann_{\T \tensor \Q} J/C =
 \Ann_{\T \tensor \Q} S_2(\Q)/ S_C =
 \Ann_{\T \tensor \Q} \SC$).
We have a perfect $\T$-equivariant bilinear pairing $\T\times
S \to\Z$ given by $(t,g)\mapsto a_1(t(g))$.\\

\noindent{\em Claim:} The induced pairing
\begin{eqnarray} \label{eqn:inducedpairing}
\T/ \ICp \times S[\IC] \ra \Z 
\end{eqnarray}
is perfect. 
\begin{proof}
We follow the proof of Lemma~\ref{perfpair4}, and the only
thing that one has to show differently is that 
the map $\T/ \ICp \ra \Hom(S[\IC],\Z)$
is injective. 
Suppose the image of~$T \in \T$ in~$\T/\ICp$ is in the kernel of this map.
Then if $h \in S[\IC]$, we have $a_1(h \divs T) = 0$.
But then $a_n(h \divs T) = a_1((h\divs T) \divs T_n) =
a_1((h\divs T_n) \divs T) =  0$ for all $n$
(considering that $h \divs T_n \in S[\IC]$),
and hence $h \divs T = 0$ for all $h \in S[\IC]$. 
Hence $T$ annihilates $S[\IC] = S \cap \SC$, and 
so it annihilates~$\SC$.
Thus $T \in \ICp$, which proves the injectivity.
\end{proof}
Using this pairing, $\Hom(\TCp/ \Ann_\TCp(A/C), \Z)$ may be viewed
as a saturated subgroup of~$S[\IC]$. Now
$\Ann_{\TCp \tensor \Q} A/C
= \Ann_{\TCp \tensor \Q} S_A/S_C = I_f/\ICp \tensor \Q
= \IB/\ICp \tensor \Q$ and
thus on tensoring with~$\Q$, $\TCp/ \Ann_\TCp(A/C)$ is dual,
under the pairing~(\ref{eqn:inducedpairing}),
to~$S[\IB]$, which is itself saturated in~$S[\IC]$.
Hence 
$$\Hom(\TCp/ \Ann_\TCp(A/C), \Z) = S[\IB].$$
Similarly, $$\Hom(\TCp/ \Ann_\TCp(B/C), \Z) = S[\IA].$$
Bearing all this in mind, and 
making the following changes in
Sections 3--4 of~\cite{ars:moddeg}: 
replace
$J$ by~$J_0(N)/C$, $A$ by~$A/C$,
$B$ by~$B/C$, $\T$ by~$\TCp = \T / \ICp$,
the proof in loc. cit. 
that the modular exponent divides the congruence exponent 
(with the changes mentioned above)
gives us the first statement in the lemma.

The proof of the second part of the lemma is a generalization of
the proof of Theorem~3.6(a) in~\cite{ars:moddeg}, which says
that if $\ell$ is a prime
such that $\ell^2 \nmid N$, then $\ell$ does not divide
the ratio of the congruence exponent to the modular exponent,
as we now explain.
Let $\T_{A/C}$ and $\T_{B/C}$ denote the images of~$\TCp$ 
acting on~$A/C$ and~$B/C$ respectively.
Let $\pi_{A/C}$ and~$\pi_{B/C}$ denote the 
maps $\TCp \ra \T_{A/C}$ and $\TCp \ra \T_{B/C}$ respectively.
Let $\RCp = \pi_{A/C}( \ker (\pi_{B/C}))$ 
and let $\SCp$ be the annihilator in~$\T_{A/C}$ of~$(A/C) \cap (B/C)$.
By a reasoning as in~\cite[\S5.1]{ars:moddeg}, except
replacing 
$J$ by~$J_0(N)/C$, $A$ by~$A/C$,
$B$ by~$B/C$, $\T$ by~$\TCp$, 
$R$ by~$\RCp$, and $S$ by~$\SCp$ in loc. cit., we get
the following:
$\RCp \subseteq \SCp$, 
the exponent
of $\frac{S[\IC]}{S[\IB] + S[\IA]}$ is the exponent of 
$\T_{A/C}/\RCp$, and multiplication by the exponent
of~$(A/C) \cap (B/C)$ annihilates $\T_{A/C}/\SCp$.
Let $\T'_{\scriptscriptstyle{C'}}$ be the saturation of~$\TCp$ in~$\End(J/C)$.
Then in a manner similar to loc. cit. (with the changes mentioned above),
we get an injection 
$\SCp/\RCp \hookrightarrow \T'_{\scriptscriptstyle{C'}}/\TCp$ 
and to prove the second part of our lemma, 
it suffices to show 
that $\TCp = \T'_{\scriptscriptstyle{C'}}$ locally
at all maximal ideals of $\TCp$ with residue characteristic~$\ell$
that contain the annihilator of~$A/C$,
which is $I_f/\ICp$
(since $\Ann_{\T \tensor \Q} A/C =
 \Ann_{\T \tensor \Q} S_A/ S_C =
 \Ann_{\T \tensor \Q} S_{[f]}$).
Considering that $\TCp = \T / \ICp$,
it suffices to show that $\T = \T'$ for
all maximal ideals of~$\T$ with residue characteristic~$\ell$
which contain~$I_f$ (where $\T'$ is the
saturation of~$\T$ in~$\End J$). But this 
is proved in~\cite[\S5.1]{ars:moddeg}. 
\end{proof}

If $g$ and $h$ are eigenforms in~$S_2(\Gamma_0(N), \C)$
and $\qq$ is an ideal in the ring of integers of a number
field that contains all the Fourier coefficients of~$g$ and~$h$,
then we will say that $g$ is congruent to~$h$ 
modulo~$\qq$ and write $g \equiv h \pmod \qq$ 
if $a_n(g) \equiv a_n(h) \pmod \qq$ for all $n \in \N$, where
as usual if $g' \in S_2(\Gamma_0(N), \C)$, then 
$a_n(g')$ denotes the $n$-th Fourier coefficient of~$g'$.

\begin{lem} \label{lem:congr}
A prime~$q$ divides the order of the group
$$\frac{S[\IC]}{S[\IB] + S[\IA]}$$ if and only if
there is a normalized eigenform~$g \in S_2(\Gamma_0(N), \C)$ with
$L(A_g,1) = 0$ such that  
$g$ is congruent to~$f$ modulo a prime ideal $\qq$
over~$q$ 
in the ring of integers
of the number field generated by the Fourier coefficients
of~$f$ and~$g$.
\end{lem}
\begin{proof}
This follows from a slight modification of~\cite[\S1]{ribet:modp} and 
we only indicate the changes that need to be made in loc.~cit.
Replacing $S$ by $S_2(\C)[\ICp]$ and letting $X = S[\IB] \tensor \C$
and $Y = S[\IA] \tensor \C$ in loc. cit., by the discussion
on p.~194-196 of loc. cit., there
exist normalized eigenforms $f' \in S[\IB] = S[I_f]$ and $g' \in S[\IA]$
such that $f' \equiv g'$ modulo 
a prime ideal over~$q$
in the ring of integers
of the number field generated by the Fourier coefficients
of~$f'$ and~$g'$.
By the definitions of $\ICp$, $\IA$, and $\IB = I_f$, 
it is clear that  $f'$ is a Galois conjugate
of~$f$ by some~$\sigma \in {\rm Gal}(\Qbar/\Q)$, and $L(g',1) = 0$.
Applying the inverse of $\sigma$, and letting $g = \sigma^{-1}(g')$,
we get the statement in the lemma.
\end{proof}

\begin{prop} \label{prop:ord_to_congr}
If an odd prime
$q$ divides ${\Mid \frac{H^+}{H[I_e]^+ + K^+} \miD}$,
then there is a normalized eigenform~$g \in S_2(\Gamma_0(N), \C)$ with
$L(g,1) = 0$ such that $g$ is congruent to~$f$ modulo 
a prime ideal over~$q$
in the ring of integers
of the number field generated by the Fourier coefficients
of~$f$ and~$g$.
\end{prop}

\begin{proof}
Note that ${\Mid \frac{H^+}{H[I_e]^+ + K^+} \miD}$ differs from
${\Mid \frac{H}{H[I_e] + K} \miD}$ only in powers of~$2$.
The proposition now follows from 
Lemmas~\ref{prop:intersect}, \ref{prop:inter_to_cong},
and~\ref{lem:congr}.
\end{proof}

\begin{prop} \label{prop:conr_divs_Lvalue}
Let $q$ be an odd prime such that $q^2 \ndiv N$.
Suppose that there is a 
normalized eigenform~$g \in S_2(\Gamma_0(N), \C)$ with
$L(g,1) = 0$ such that $g$ 
is congruent to~$f$
modulo a prime ideal over~$q$
in the ring of integers
of the number field generated by the Fourier coefficients
of~$f$ and~$g$. 
Then $q$ divides ${\Mid \frac{H^+}{H[I_e]^+ + K^+} \miD}$.
Suppose moreover that $q$ does not divide $\Mid A_f(\Q)_{\rm tor} \miD$.
Then 
$q$ divides $\LAf(1) / \OAf$, and in particular,
$\frac{\LAf(1)}{\OAf} \equiv 
\frac{{{L_{\scriptscriptstyle{A_g}}}}(1)}
{{\Omega_{\scriptscriptstyle{A_g}}}} \bmod q$.
\end{prop}

\begin{proof}

The first part of the proposition
follows from Lemma~\ref{prop:intersect}, the second statement in
Lemma~\ref{prop:inter_to_cong}, and the fact that
${\Mid \frac{H^+}{H[I_e]^+ + K^+} \miD}$ differs from
${\Mid \frac{H}{H[I_e] + K} \miD}$ only in powers of~$2$.
The second part follows from 
Theorem~\ref{thm:mainform}, in view of the following facts:  
${\Mid \pi_*(\T  e) / \pi_*(\Im  e) \miD}$ divides 
$\Mid A_f(\Q)_{\rm tor} \miD$ (by Lemma~\ref{lem:divs}),
$q \ndiv c_\infty(A_f)$
(since $q$ is odd), $q \ndiv \caf$ 
(by~\cite{agst:manin}, considering that $q^2 \ndiv N$),
and $L_{\scriptscriptstyle{A_g}}(1) =0$ (since $L(g,1) = 0$).
\end{proof}

\comment{
\begin{prop} \label{prop:conr_divs_Lvalue}
Suppose $q$ is an odd prime such that $q^2 \ndiv N$,
and $f$ is congruent modulo a prime ideal over~$q$ to 
a normalized eigenform~$g \in S_2(\Gamma_0(N), \C)$ such that
$L(A_g,1) = 0$. Then $q$ divides ${\Mid \frac{H^+}{H[I_e]^+ + K^+} \miD}$.
If moreover $q$ does not divide $\Mid A_f(\Q)_{\rm tor} \miD$, then 
$q$ divides $\LAf(1) / \OAf$.
\end{prop}

\begin{proof}
It is proved in~\cite{ars:moddeg}) that if a prime~$\ell$
divides the congruence number, but $\ell^2 \nmid N$, 
then $\ell$ divides the
modular number.
It suffices to show $R_C = S_C$, i.e., $\T_\C = \T'_C$ locally
at max ideal of $\T_C / \im(I_A)$.
Now $A_f \ra J_A/ J_C$ is an isogeny, so $\Ann_{\T_C} J_A/ J_C
= \im(I_A) = (I_A + I_C)/I_C$
is the image of $\Ann_\T A_f = I_f$. So suffices to check
if $I_f \subseteq \m$, etc.

By a simple generalization of the proof,
it follows that if $q$ is as in the first statement of the
Proposition, then $q$ divides
$\mid (A_1^{\vee}/A_0^{\vee}) \cap (A_2^{\vee}/A_0^{\vee}) \mid$
(notation as in Lemma~\ref{prop:intersect}). The first part
of the proposition
now follows from Lemma~\ref{prop:intersect} and the fact
that $q$ is odd. The second part
follows from 
Prop.~\ref{thm:mainform}, in view of the following facts:  
${\Mid \pi_*(\T  e) / \pi_*(\Im  e) \miD}$ divides 
$\Mid A_f(\Q)_{\rm tor} \miD$, $q \ndiv c_\infty(A_f)$
(since $q$ is odd), and $q \ndiv \caf$ 
(since $q^2 \ndiv N$, by~\cite{agst:manin}). 
\end{proof}
}

\begin{cor} \label{cor:mwrk}
Let $q$ be an odd prime such that $q^2 \ndiv N$.
Suppose that there is a 
normalized eigenform~$g \in S_2(\Gamma_0(N), \C)$ such that 
$A_g$ has positive Mordell-Weil rank and $g$ 
is congruent to~$f$
modulo a prime ideal over~$q$
in the ring of integers
of the number field generated by the Fourier coefficients
of~$f$ and~$g$. 
Then the second part of the Birch and Swinnerton-Dyer conjecture
implies that $q$ divides 
$\Mid \Sha(A_f) \miD \cdot \prod_{p\mid N} c_p(A_f)$.
If moreover, $N$ is prime,  then the conjecture implies that
$q$ divides $\Mid \Sha(A_f) \miD$. 
\end{cor}

\begin{proof}
By~\cite{kollog:finiteness}, since $A_g$ has positive Mordell-Weil rank,
$L(g,1) = 0$. Hence by Proposition~\ref{prop:conr_divs_Lvalue},
$q$ divides ${\Mid \frac{H^+}{H[I_e]^+ + K^+} \miD}$.
The first conclusion of the Corollary now follows from
formula~(\ref{sqfreelevelShah}) and Lemma~\ref{lem:divs},
and the second conclusion (for $N$ prime) follows from
Proposition~\ref{prop:divs}.
\end{proof}

We also have the following characterization of
certain primes that divide the factor
${\Mid \frac{H^+}{H[I_e]^+ + K^+} \miD}$.

\begin{cor} \label{cor:divs}
Suppose $q$ is an odd prime such that $q^2 \ndiv N$.
Then $q$ divides ${\Mid \frac{H^+}{H[I_e]^+ + K^+} \miD}$
if and only if 
there is a normalized eigenform~$g \in S_2(\Gamma_0(N), \C)$ with
$L(g,1) = 0$ such that  $g$ is congruent to~$f$ modulo 
a prime ideal over~$q$
in the ring of integers
of the number field generated by the Fourier coefficients
of~$f$ and~$g$.
\end{cor}
\begin{proof}
This follows from Propositions~\ref{prop:ord_to_congr}
and~\ref{prop:conr_divs_Lvalue}.
\end{proof}

\section{Visibility: relating the congruence to 
the Shafarevich-Tate group and the component group}
\label{section:visibility}

In this section, we finally relate the factor 
$\Mid \frac{H^+}{H[I_e]^+ + K^+} \miD$ in the numerator
of our formula~(\ref{formula:lovero}) for~$\LAf(1)/\OAf$
to the numerator
${\Mid \Sha(A_f) \miD \cdot \prod_{\scriptscriptstyle{p |N}}  c_p(A_f)}$
of the Birch and Swinnerton-Dyer
conjectural formula~(\ref{bsdform}) for~$\LAf(1)/\OAf$.
Recall that since $L(A_f,1) \neq 0$, 
by~\cite{kollog:finiteness}, $\Sha(A_f)$ is finite and
$A_f(\Q)$ has rank zero.
At several instances below, we will consider the torsion 
group $A_f[\qq]$ where $\qq$ is a prime ideal in a number
field~$L$
containing the Fourier coefficients of~$f$ and another
eigenform -- this makes sense if $A_f$ has the full action
of the ring of integers of~$L$. If not, then 
we merely replace $\qq$ by a prime of fusion,
which is a maximal ideal of~$\T$ (see, e.g., \cite[p.196]{ribet:modp}
and~\cite[\S2]{eki:congr})
that is contained in~$\qq$, and all the arguments go through.
We have preferred to consider prime ideals of~$L$
instead of maximal ideals of~$\T$ simply to keep the notation 
similar to the one
in the article~\cite{dsw}, which we will often refer to in this section.
If $p$ is a prime that divides~$N$, then we denote by~$w_p$
the eigenvalue of the Atkin-Lehner involution~$W_p$ acting on~$f$.~
Similarly $w_N$ denotes 
the eigenvalue of the Atkin-Lehner involution~$W_N$ acting on~$f$
(so that $w_N$ is the product over all primes that divide~$N$ 
of the~$w_p$'s).
\begin{thm} \label{thm:mainprime}
Suppose $N$ is prime, and 
$q$ is a prime such that 
$q$ divides $\Mid \frac{H^+}{H[I_e]^+ + K^+} \miD$, 
and $q \nmid N(N-1)$.
Assume that for all newforms~$g$ on~$\Gamma_0(N)$,
if $L(g,1) = 0$, then
the rank of~$A_g(\Q)$ is positive.
Then $q$ divides~$\Mid \Sha(A_f) \miD$. If moreover we assume
the parity conjecture, then $q^2$ divides~$\Mid \Sha(A_f) \miD$
(both the assumption on the rank and the parity
conjecture hold if the first part of the Birch and
Swinnerton-Dyer conjecture (Conjecture~\ref{bsd1}) is true).
\end{thm}

\begin{proof}
By Prop.~\ref{prop:ord_to_congr}, 
there is a normalized eigenform~$g \in S_2(\Gamma_0(N), \C)$ such that
$L(A_g,1) = 0$, and $f$ is congruent to~$g$ modulo 
a prime ideal~$\qq$ over~$q$
in the ring of integers
of the number field generated by the Fourier coefficients
of~$f$ and~$g$.
By the hypothesis, the rank of~$A_g(\Q)$ is positive.
By~\cite[Prop.~14.2]{mazur:eisenstein}, 
since $q \ndiv 2 \cdot {\rm numr}(\frac{N-1}{12})$, 
it follows that $A_f^\vee[\qq]$ and $A_g^\vee[\qq]$ are  irreducible.
Also, since $\LAf(1) \neq 0$, we have $w_N = -1$. Hence
by hypothesis, $N \not\equiv - w_N \pmod q$. Thus the hypotheses
of Theorem~6.1 of~\cite{dsw} are satisfied
and the conclusion of this theorem tells us that
$q$ divides~$\Mid \Sha(A_f^\vee) \miD$ 
(in the notation of~\cite{dsw}, $r \geq 1$ since
the dimension of $H^1_f(\Q,V'_\qq(1))$ is an upper bound for the
rank of~$A_g(\Q)$; also
given that  
$\Sha(A_f^\vee)$ is finite,
the $q$-primary part of~$\Sha(A_f^\vee)$ is the same as $H^1_f(\Q, A_\qq(1))$).
By the perfectness of the Cassels-Tate pairing,
$q$ divides~$\Mid \Sha(A_f) \miD$ as well.
This proves the first assertion of the theorem.

Since $q$ is odd, $f$ and~$g$ have the same eigenvalue
under the Atkin-Lehner involution, and hence the same sign
in their functional equations (see, e.g., Remark~\ref{rmk:cp}). 
Thus if we assume the parity
conjecture, then $A_g^\vee$ has even Mordell-Weil rank, and so
$r \geq 2$ (where $r$ is as in the statement of Theorem~6.1 of~\cite{dsw}).
Theorem~6.1 of~\cite{dsw} in fact tells us that
$q^r$ divides~$\Mid \Sha(A_f^\vee) \miD$, which gives the 
second assertion of our theorem.
\end{proof}

As mentioned towards the end of Section~\ref{section:extract},
the second part of the Birch and Swinnerton-Dyer
conjecture implies that when $N$ is prime, if an odd prime~$q$ divides
${\Mid \frac{H^+}{H[I_e]^+ + K^+} \miD}$, then $q$ 
divides~$\Mid \Sha(A_f) \miD$.
Thus Theorem~\ref{thm:mainprime} may be viewed as a partial
result towards the second part of the Birch and Swinnerton-Dyer
conjecture.

\begin{eg} \label{eg:389} 
Consider the prime level~$389$.
From~\cite{stein:modformsdata}, one sees
that the quotient associated the newform denoted~{\bf 389E1}
is the winding
quotient~$J_0(N)/I_e J_0(N)$ for~$N=389$ (the other newforms
of level~$389$ have positive analytic rank).
The factor~$\Mid \frac{H^+}{H[I_e]^+ + K^+} \miD$
was computed for this winding quotient in~\cite[\S3.3]{agashe:phd}
and its odd part was found to be~$5$. Moreover, $5 \ndiv 389(389-1)$, and 
hence if one assumes the first part of the Birch and Swinnerton-Dyer
conjecture, Theorem~\ref{thm:mainprime} implies that 
$5$ divides $\Sha({\bf 389E})$. 
In fact, it turns out that one can show $5$ divides
$\Sha({\bf 389E})$
without assuming the conjecture -- see~\cite[\S4.1]{agst:vis}.
\end{eg}

The following proposition is easily extracted 
from~\S7.4 of~\cite{dsw}; see also~\cite[\S7]{dummigan:level}.

\begin{prop}[Dummigan-Stein-Watkins] \label{prop:cp}
Let $p$ be a prime such that $p || N$. Suppose that there is a newform
$h$ of level dividing~$N/p$, 
and an odd prime~$q \neq p$ such that 
modulo a prime ideal~$\qq$ over~$q$
in some number field containing the Fourier coefficients
of~$f$ and~$h$, we have $f \equiv h$ 
(for Fourier coefficients of index coprime to~$Nq$).
Suppose $A_f[\qq]$ and~$A_h[\qq]$ are irreducible and that
$p \not\equiv - w_p \pmod q$. If $w_p = -1$, then
$\ord_{\qq}(c_p(A_f))>0$.
\end{prop}

\begin{proof}
The proof is essentially given in the 
first example in~\S7.4 of~\cite{dsw}; 
we repeat the argument here for clarity. 
We use the notation as in~\cite{dsw}. 

Let $T''_\qq$, $V''_\qq$, and $A''_\qq$ be the objects attached to~$h$,
just as $T_\qq$, $V_\qq$, and $A_\qq$ have been attached to~$f$.
Since $p^2 \ndiv N$, $p$ does not divide the level of~$h$, 
and so the representation $V''_{\qq}$ is unramified at $p$. 
Hence $H^0(I_p,A''[\qq])$ is
two dimensional. Since $A_f[\qq]$ and~$A_h[\qq]$ are irreducible,
by the Chebotarev density theorem, they are isomorphic as Galois
representations and thus $H^0(I_p,A[\qq])$ is also two dimensional. 
Since $p || N$, by~\cite{carayol:hilbert},
$H^0(I_p,V_\qq)$ is one dimensional, and  
following the argument in Case~(2) of the proof
of Theorem~6.1 of~\cite{dsw}, 
we see that the eigenvalue of $\Frob_p^{-1}$ acting on
$H^0(I_p,V_\qq)$ is $\alpha=-w_p p^{(k/2)-1}$ (for us, $k=2$).
Hence $\Frob_p^{-1}$  acts as $\alpha=-w_p p^{(k/2)-1}$
on the subspace of $H^0(I_p,A[\qq])$ formed by 
the image of $H^0(I_p,V_\qq)$.
Since $\alpha\beta=p^{k-1}$, 
the other eigenvalue must be $\beta=-w_pp^{k/2}$. 
Twisting by $k/2$, we see that the corresponding eigenvalues of 
$\Frob_p^{-1}$ on~$H^0(I_p,A[\qq](k/2))$ are $-w_p$ and $-w_p/p$. Since
$p \not\equiv - w_p \pmod q$ and $w_p=-1$,
the two eigenvalues are different. Hence 
the quotient of $H^0(I_p,A[\qq](k/2))$ by the image of $H^0(I_p,V_{\qq}(k/2))$
may be viewed as a subspace, and 
$\Frob_p^{-1}$ acts as $-w_p$ on the quotient.
As $w_p = -1$, we have $\ord_{\qq}(c_p(k/2))>0$.
Since $k=2$, the prime-to-$p$ parts of $c_p(k/2)$ (as 
defined in~\cite[\S4]{dsw}) and $c_p(A_f)$ are the same (see the Appendix).
\end{proof}

\begin{rmk} \label{rmk:cp}
In Proposition~\ref{prop:cp}, $f$ need not have analytic rank~$0$.
Note that if $w_p = -1$, then the analytic rank of~$h$ has a different
parity than that of~$f$. This is because the sign 
in the functional equation of a newform (of weight~$2$) is the negative 
of the product of the eigenvalues of~$W_p$ over all primes $p$ dividing 
the level of the form in question.
If this sign is positive, then the analytic rank of the newform
is even, and 
it is odd otherwise (see, e.g., \cite[Ex.~4.3.3]{silverberg:openque}).
\end{rmk}

\begin{prop}  \label{thm:vissqfree}
Let $q$ be an odd prime such that \mbox{$q \ndiv N$}.
Suppose $g \in S_2(\Gamma_0(N), \C)$ is an eigenform (for
all the Hecke operators) such that
$A_g(\Q)$ has positive rank 
and $f$ is congruent to~$g$ modulo a prime ideal~$\qq$ over~$q$
in the ring of integers
of the number field generated by the Fourier coefficients
of~$f$ and~$g$. 
Suppose that $A_f[\qq]$ is an irreducible representation
of the absolute Galois group of~$\Q$. 
Assume that for all primes $p \divs N$, 
$p \not\equiv - w_p \pmod q$.
We have two possibilities:

\noindent
Case (i) For all primes $p \divs N$,
$f$ is not congruent modulo~$\qq$ to a newform
of level dividing $N/p$ (for Fourier coefficients of index coprime to~$Nq$):\\
In this case, suppose that for all primes~$p$ such that $p^2 \divs N$, 
we have $p \not\equiv -1  \pmod q$. 
Then $q$ divides $\Mid \Sha(A_f) \miD$. 
If moreover we assume
the parity conjecture, then $q^2$ divides~$\Mid \Sha(A_f) \miD$.

\noindent Case (ii) $f$ is congruent modulo~$\qq$ to a newform of lower level
(for Fourier coefficients of index coprime to~$Nq$):\\
In this case, suppose that \\
(*) there is a prime $p$ dividing~$N$ 
such that 
$f$ is congruent modulo~$\qq$ 
to a newform~$h$ of level dividing $N/p$ 
(for Fourier coefficients of index coprime to~$Nq$), with
$p^2 \ndiv N$, $w_p=-1$, and
$A_h[\qq]$ irreducible.\\
Then $q$ divides $\prod_{\scriptscriptstyle{{p\mid N}}} c_p(A_f)$.
\end{prop}

\begin{proof} 
The first part follows from~\cite[Thm.~6.1]{dsw}
(cf. the proof of Theorem~\ref{thm:mainprime} for details),
and the second part follows from Proposition~\ref{prop:cp}.
\end{proof}

In view of Corollary~\ref{cor:mwrk}, the Proposition above provides
theoretical evidence towards the second part of the Birch
and Swinnerton-Dyer conjecture.

\begin{rmk} 
(1) In condition~(*),  the statement 
``$f$ is congruent $\bmod \ \qq$ 
to a newform of level dividing $N/p$ 
(for Fourier coefficients of index coprime to~$Nq$)''
may be replaced by the potentially weaker statement  
``$A[\qq]$ is unramified at~$p$'' (this can be seen
from the proof of Prop.~\ref{prop:cp}).
\\
(2) 
Suppose 
$h$ is a newform of level~$N/p$ for some prime~$p || N$ with $w_p=1$
such that $f \equiv h$ modulo a prime ideal with odd residue characteristic.
This is one situation where the condition~(*) does not hold.
Then, as N. Dummigan pointed out to us, by the reasoning in Remark~\ref{rmk:cp},
the sign in the functional equation of $L(A_h,s)$
is the same as that of $L(A_f,s)$, i.e., it is positive.
Thus $A_h$ has even analytic rank. If the analytic rank of~$h$
is in fact zero, then   
since $h \equiv g \pmod \qq$, one expects that $\Sha(A_h)$ is non-trivial,
and hence $h$ should show up in the tables in~\cite{agst:bsd}.
With this in mind, we searched the tables 
in~\cite{agst:bsd} 
and found only one potential example where (*) may fail because
the condition on~$w_p$ does not hold:
At level $N=1994$, whose prime factorization is $2 \cdot 997$, the
newform~$f = {\bf 1994D}$ is of rank~$0$ with $w_2=1$
($w_{997}=-1$, but there are no cusp forms over~$\Gamma_0(2)$).
The newform~$f$ is congruent
modulo a prime ideal over~$3$ to $g = {\bf 997B}$ of positive 
analytic rank (this can be deduced from the table in~\cite{agst:bsd}
and the result of~\cite{ars:moddeg}
that the modular exponent divides the congruence exponent).
Similarly, one finds that $g$ is congruent modulo a prime ideal over~$3$
to $h = {\bf 997H}$ of Mordell-Weil rank zero. However, we do not know if the 
congruences mentioned in the previous two sentences 
hold modulo the {\em same} ideal over~$3$.
\end{rmk}

\begin{thm}\label{thm:mainsqfree}
Let $q$ be a prime such that 
$q$ divides $\Mid \frac{H^+}{H[I_e]^+ + K^+} \miD$.\\
Suppose that
$q \ndiv 2 N$ and that for all maximal ideals~$\qq$ of~$\T$
with residue characteristic~$q$, $A_f[\qq]$ is irreducible.
Assume that for all newforms~$g$ of level dividing~$N$,
if $L(g,1) = 0$, then the rank of~$A_g(\Q)$ is positive
(this would hold if the first part of the Birch and Swinnerton-Dyer conjecture
(Conjecture~\ref{bsd1}) is true).
Suppose that  
for all primes $p \divs N$, $p \not\equiv - w_p \pmod q$.
We have two possibilities:


\noindent
Case (i) For all primes $p \divs N$,
$f$ is not congruent modulo a maximal ideal of~$\T$ over~$q$ to a newform
of level dividing $N/p$ (for Fourier coefficients of index coprime to~$Nq$):\\
In this case, suppose that for all primes~$p$ such that $p^2 \divs N$, 
we have $p \not\equiv -1  \pmod q$. 
Then $q$ divides $\Mid \Sha(A_f) \miD$. If moreover we assume
the parity conjecture, then $q^2$ divides~$\Mid \Sha(A_f) \miD$
(the parity conjecture hold if the first part of the Birch and
Swinnerton-Dyer conjecture (Conjecture~\ref{bsd1}) is true).


\noindent Case (ii) For some prime $p$ dividing~$N$, $f$ is congruent 
modulo a maximal ideal of~$\T$ over~$q$ to a newform
of level dividing $N/p$ (for Fourier coefficients of index coprime to~$Nq$):\\
In this case, suppose that \\
(*) there is a prime $p$ dividing~$N$ and a maximal ideal~$\qq$ of~$\T$
over~$q$ 
such that 
$f$ is congruent modulo~$\qq$ 
to a newform~$h$ of level dividing $N/p$ 
(for Fourier coefficients of index coprime to~$Nq$), with
$p^2 \ndiv N$, $w_p=-1$, and
$A_h[\qq]$ irreducible.\\
Then $q$ divides $\prod_{\scriptscriptstyle{{p\mid N}}} c_p(A_f)$.
\comment{
\noindent (i) If for all primes $p \divs N$,
$f$ is not congruent modulo a prime ideal over~$q$ to a newform
of level dividing $N/p$ (for Fourier coefficients of index coprime to~$Nq$),
then $q$ divides $\Mid \Sha(A_f) \miD$. \\
\noindent (ii) If statement (*) in Proposition~\ref{thm:vissqfree} is false
for all maximal ideals~$\qq$ of~$\T$
with residue characteristic~$q$,
then $q$ divides $\Mid \Sha(A_f) \miD \cdot 
\prod_{\scriptscriptstyle{p\mid N}} c_p(A_f)$.
}
\end{thm}

\begin{proof}
By Prop.~\ref{prop:ord_to_congr}, 
there is a normalized eigenform~$g \in S_2(\Gamma_0(N), \C)$ such that
$L(g,1) = 0$, and $f \equiv g$ modulo 
a prime ideal~$\qq$ over~$q$
in the ring of integers
of the number field generated by the Fourier coefficients
of~$f$ and~$g$.
By the hypothesis, the rank of~$A_g(\Q)$ is positive.
The theorem then follows by Proposition~\ref{thm:vissqfree}.
\end{proof}

As mentioned towards the end of Section~\ref{section:extract},
the second part of the Birch and Swinnerton-Dyer
conjecture implies
that the part of $\Mid \frac{H^+}{H[I_e]^+ + K^+} \miD$ 
that is coprime to~$| A_f(\Q) |$ divides
${\Mid \Sha(A_f) \miD \cdot \prod_{\scriptscriptstyle{p |N}}  c_p(A_f)}$
up to a power of~$2$ and powers of primes whose squares divide~$N$.
Thus Theorem~\ref{thm:mainsqfree}(ii) is a partial result towards
the second part of the Birch and Swinnerton-Dyer
conjecture.

\begin{rmk} \label{rmk:nontriv}
The factor ${\Mid \frac{H^+}{H[I_e]^+ + K^+} \miD}$
is often non-trivial.
In~\cite{agst:bsd}, one finds a table of all quotients associated
to newforms~$f$ of level $N < 2333$ such that $L(f,1) \neq 0$ and
the conjectured value of $\Mid \Sha(A_f) \miD$ is bigger than~$1$.
In all the cases when the entry under the column labelled~``B''
is {\em not} NONE, there is an odd prime~$q$
and a normalized 
eigenform~$g \in S_2(\Gamma_0(N), \C)$ such that
$L(g,1) = 0$, and $f \equiv g$ modulo a prime ideal lying over~$q$.
If $q^2 \ndiv N$, then by
Prop.~\ref{prop:conr_divs_Lvalue}, $q$ divides
${\Mid \frac{H^+}{H[I_e]^+ + K^+} \miD}$. Thus all such
entries in the table give examples where the factor
${\Mid \frac{H^+}{H[I_e]^+ + K^+} \miD}$ is non-trivial,
and where a prime dividing this factor divides the
conjectural order of~$\Sha(A_f)$.
The first level where this happens (also the first level where 
an odd prime divides the Birch and Swinnerton-Dyer conjectural
order of~$\Sha(A_f)$) is~$389$, where
there is a newform quotient of dimension~$20$
for which $5$ divides ${\Mid \frac{H^+}{H[I_e]^+ + K^+} \miD}$.
When the entry under the column labelled~``B''
is NONE, the factor ${\Mid \frac{H^+}{H[I_e]^+ + K^+} \miD}$
is a (possibly trivial) power of~$2$.
\end{rmk}

\begin{eg}  \label{eg:cp}
We now give an example where a prime divides
the factor $\Mid \frac{H^+}{H[I_e]^+ + K^+} \miD$,
but does not divide the conjectured order of~$\Sha(A_f)$,
and instead divides $\prod_{\scriptscriptstyle{p |N}}  c_p(A_f)$.
This example was obtained using
W. Stein's ``The Modular Forms Explorer''~\cite{stein:modformsdata}
and we use the notation as therein. 
The newform~$f = {\bf 1751C1}$ is of analytic rank~$0$
and 
the newform~$h = {\bf 103A1}$ has positive analytic rank
(note that $1751 = 17 \cdot 103$).
From \cite[Table~2]{agst:bsd}, we find that 
$505$ divides the order of the intersection of
$A_f^{\vee}$ with the abelian subvariety of~$J_0(1751)$ generated
by the images of~$A_h^{\vee}$ under the degeneracy maps. 
Hence, by 
the result in~\cite{ars:moddeg} that the modular exponent
divides the congruence exponent, we conclude that 
$f$ is congruent modulo a prime ideal over~$101$
to an eigenform~$g$ in the subspace
generated by~$h$ and $B_{17}(h)$ 
(a similar result holds for the congruence prime~$5$). 
We conclude from Proposition~\ref{prop:conr_divs_Lvalue}
that $101$ divides ${\Mid \frac{H^+}{H[I_e]^+ + K^+} \miD}$.
One finds that $w_{17} = -1$, and 
since $103$ is prime and $101 \ndiv {\rm numr}(\frac{103-1}{12})$,
$A_f[\qq] \isom A_h[\qq]$ is irreducible.
Hence, Theorem~\ref{thm:vissqfree} implies that $101$ divides~$c_{17}(f)$
(a similar result holds for the congruence prime~$5$
replacing~$101$ as well).
As a confirmation, 
one finds from~\cite{stein:modformsdata} that $c_{17}(A_f) = 6565$
(which is divisible by~$101$ and~$5$).
It turns out that $5$ divides the conjectural value
of~$\Mid \Sha(A_f) \miD$, but $101$ doesn't. 
Note that neither $5$ nor~$101$ divides the order of
either~$A_f(\Q)_{\rm tor}$ or~$\Afdual(\Q)_{\rm tor}$.
\end{eg}

\begin{rmk} \label{rmk:egs}
(1) By the comment just after the proof
of Theorem~\ref{thm:mainsqfree}, 
the Birch and Swinnerton-Dyer conjecture suggests
that the extra hypothesis~(*) is unnecessary
in Proposition~\ref{thm:vissqfree} and Theorem~\ref{thm:mainsqfree}.
There are other versions of ``visibility theorems''
similar to Theorem~\ref{thm:vissqfree}. In the version 
in~\cite{agst:vis} and that of Cremona and Mazur (see
the appendix of~\cite{agst:bsd}), instead of our hypothesis~(*)
in Proposition~\ref{thm:vissqfree} 
there is the hypothesis 
that $q$ does not divide~$\prod_{\scriptscriptstyle{p\mid N}} c_p(A_g)$. 
It is not clear to us whether the two hypotheses 
are related (there is some difference in the other hypotheses 
of the theorems as well). \\
(2) In~\cite{agst:bsd}, whenever an odd prime~$q$ divides
${\Mid \frac{H^+}{H[I_e]^+ + K^+} \miD}$ (cf. Remark~\ref{rmk:nontriv})
and the entry
under the column labelled~``Vis'' is a number, one finds that $q^2$
divides~$\Mid \Sha(A) \miD$. 
Note that in most of these cases, we do not expect
$q^2$ to divide ${\Mid \frac{H^+}{H[I_e]^+ + K^+} \miD}$
(roughly speaking, locally the homology is often of rank two
over the Hecke algebra, and our factor captures only half
of the homology; see also Example~\ref{eg:389}).
This seems to suggest that if an odd prime~$q$ divides
${\Mid \frac{H^+}{H[I_e]^+ + K^+} \miD}$, then it also
divides the other factor 
$\Mid \frac{H[I_e]^+}{\Im  e + H[I_e]^+ \cap K^+} \miD$
in~(\ref{mainform}) (perhaps under some additional mild hypotheses).  
\\
(3) 
We have not said much about the factor
$\Mid \frac{H[I_e]^+}{\Im  e + H[I_e]^+ \cap K^+} \miD$, except
the remark just above, which indicates that this factor should be 
divisible by primes~$q$ such that~$f$ is congruent modulo a prime ideal over~$q$
to an eigenform in~$S_2(\Gamma_0(N),\C)$
with positive analytic rank.
However, these are not all the primes that divide this factor
in general.
In most of the examples 
in~\cite{agst:bsd} where the entry under the column~B is
NONE, this factor is non-trivial and is 
divisible by primes~$q$ such that 
there is no eigenform~$g$ congruent to~$f$ modulo a prime over~$q$
with $L(g,1) = 0$. 
It is our guess, however, that the factor
$\Mid \frac{H[I_e]^+}{\Im  e + H[I_e]^+ \cap K^+} \miD$ can be
explained by ``visibility at higher level'', as we now explain.
It can be shown that the non-trivial elements of~$\Sha(\Afdual)$ 
(whose order is the same as that of~$\Sha(A_f)$)
whose existence is implied by Theorems~\ref{thm:mainprime}
and~\ref{thm:mainsqfree} are in the kernel of the natural map
$\Sha(\Afdual) \ra \Sha(J_0(N))$ (this can be seen from the proof 
of~\cite[Thm.~6.1]{dsw}, or better still, from the analogous Theorem~3.1
of~\cite{agst:vis}). In general, if $J$ is an abelian variety
with a map $\Afdual \ra J$, then we say that an element of~$\Sha(\Afdual)$
is {\em visible in~$J$} if it is in the kernel of the induced map
$\Sha(\Afdual) \ra \Sha(J)$. 
If an element of~$\Sha(\Afdual)$ is visible in~$J_0(N)$ under
the natural inclusion $\Afdual \ra J_0(N)$, we say
that it is {\em visible at the same level}. Now one has certain
natural degeneracy maps $J_0(N) \ra J_0(NM)$ for every positive integer~$M$.
We say that an element of~$\Sha(\Afdual)$ is {\em visible at higher level}
if it is visible in~$J_0(NM)$ under a map $\Afdual \ra J_0(N) \ra J_0(NM)$
for some integer~$M$ bigger than one.
In~\cite[\S4.2]{agst:vis}, there is an example of a particular~$f$
for which an element of~$\Sha(\Afdual)$ is not visible at the same level,
but becomes visible at a higher level.
In fact, it has been conjectured that any element 
the Shafarevich-Tate group can be explained by 
visibility at some higher level (see~Conjecture~7.1.1 in~\cite{jetchev-stein}
for details and a precise statement). 
Thus there is hope that the factor
$\Mid \frac{H[I_e]^+}{\Im  e + H[I_e]^+ \cap K^+} \miD$
may be interpreted by considerations of visibility at 
the same and higher levels.
This also suggests that while our formula for $\LAf(1)/\OAf$ was 
obtained via the parametrization of~$A_f$ by~$J_0(N)$, perhaps one
should try to prove and use a similar formula for $\LAf(1)/\OAf$ 
obtained via a parametrization of~$A_f$ by~$J_0(NM)$ for an
integer~$M$.
\end{rmk}

\section{Appendix: component groups} \label{sec:app}
The proofs of some of the results in Section~\ref{section:visibility}
employ the language of representations (\`a la~\cite{bloch-kato})
as opposed to the language of abelian varieties. In this appendix
we show that if $p$ is a prime, then
the definitions of the prime to~$p$ part of the
component group at~$p$ coincide in the two languages. 
This is well known
to experts, and our aim is simply to provide some details that
we could not find in the literature. 

Let $A$ be an abelian variety and $\NerA$ its N\'eron model.
Let $\NerA^0$ denote the largest open subgroup scheme
of~$\NerA$ in which all the fibers are connected.
Let $p$ be a prime.
In Section~\ref{section:intro}, we had defined
$c_p(A) = [\NerA_p(\F_p): \NerA^0_p(\F_p)]$,
where the subscript~$p$ denotes the special fiber at~$p$.

Let  $B$ be a group and $\ell$ is a prime. 
Consider the inverse system $\{B[\ell^n]\}_{n \in \N}$ 
where for each~$n$, the map $B[\ell^{n+1}] \ra B[\ell^n]$ is multiplication
by~$\ell$. We denote by $\Tl B$ the inverse limit of the this
system, which is a $\Zl$-module in a natural way.
Let $\Vl B  = \Tl B \tensor \Ql$, and
$\Wl B = \Vl B / \Tl B$. 
The natural map $\Tl B \tensor \Ql
\ra \Tl B \tensor \Ql/\Zl$ is surjective, with kernel $\Tl B$;
hence we have a canonical isomorphism $\Wl B \isom \Tl B \tensor \Ql/\Zl$.

For simplicity of notation, for the case $B = A(\Qbar)$,
we simply write $\Tl$, $\Vl$, and $\Wl$ for the corresponding 
objects.
In~\cite[\S4]{dsw}, the authors define an integer~$c_p(1)$
associated to~$A$ as follows:
if $I_p$ denotes the inertia subgroup of
the absolute Galois group of~$\Q_p$ and $\ell \neq p$ is a prime, then
\begin{eqnarray} \label{dswcp}
\ord_\ell(c_p(1)) = \# {H^0(\Qp, \Wl^\Ip)}
 - \# {H^0(\Qp, \Vl^\Ip/ \Tl^\Ip)}, 
\end{eqnarray}
(note that in their notation
$T_\ell(1) = \Tl A$, as mentioned at the end of~\S1 in~\cite{dsw}).
The definition of $\ord_p(c_p(1))$ is more complicated, and we
shall not be concerned with it.
Our first goal in this appendix is to show
that if $\ell \neq p$ is a prime, then
$\ord_\ell(c_p(1)) = \ord_\ell (c_p(A))$ (this is used
in the proof of Prop.~\ref{prop:cp}).
The other goal of this appendix is to give some details of
how the definition in~(\ref{dswcp})
comes naturally from the formulation
of the Bloch-Kato conjecture~\cite{bloch-kato}.

Henceforth, $\ell \neq p$ is a prime. 
The following lemma is well known:

\begin{lem} \label{lemma:qlzl}
Suppose the inverse system $\{B[\ell^n]\}_{n \in \N}$ is surjective,
i.e., the multiplication
by~$\ell$ maps $B[\ell^{n+1}] \ra B[\ell^n]$ are surjective
for all~$n$. Then one has a canonical isomorphism (of groups)
$\Tl B \tensor \Ql/\Zl = B[\ell^\infty]$.
\end{lem}

\begin{proof}
Our proof is inspired by~\cite[IX.11]{grothendieck:monodromie}.
Consider the direct system $\{\Zl / \ell^n \Zl\}_{n \in \N}$,
where the maps $\Zl / \ell^n \Zl \ra \Zl / \ell^{n+1} \Zl$
are multiplication by~$\ell$. The direct limit 
$\dirlimn \Zl / \ell^n \Zl$ of this system
is the direct sum $\oplus (\Zl / \ell^n \Zl)$ modulo the subgroup
generated by all elements of the form
$(\ldots, 0, x_n, 0, \ldots, 0, x_m, 0, \dots)$, 
such that $x_m = - \ell^{m-n} x_n$,
where $m > n$ and $x_n$ and~$x_m$ are the entries in the $n$-th
and $m$-th position respectively. Then 
the assignment $(x_n) \in \oplus (\Zl / \ell^n \Zl)$ maps to
$\sum \frac{x_n}{\ell^n}$ gives a canonical isomorphism 
$\dirlimn \Zl / \ell^n \Zl \isom \Ql / \Zl$.

Next, let $(t_i)_{i \in \N}$ be an element of $\Tl B$.
Let $m$ be any positive integer. Using the construction of
the inverse limit and the fact that 
$\{B[\ell^n]\}$ is a surjective system, one sees that  
the assignment $(t_i) \tensor 1 \mapsto t_m$ gives
a canonical isomorphism $\Tl B \tensor (\Zl / \ell^m \Zl) \isom B[\ell^m]$.
This maps the direct system $\{\Tl B \tensor (\Zl / \ell^m \Zl)\}_{m \in \N}$
with the multiplication by~$\ell$ map on the second component and 
identity on the first
isomorphically to the direct system $\{B[\ell^m]\}$ where the maps
are the natural
inclusion maps $B[\ell^m] \hookrightarrow B[\ell^{m+1}]$.

Thus we see that $\Tl B \tensor \Ql/\Zl = 
\dirlimm \Tl B \tensor \Zl / \ell^m \Zl
= \dirlimm B[\ell^m] = B[\ell^\infty]$.
\end{proof}

We now apply this to the case $B = A(\Qbar)$.
Note that since $\ell \neq p$, 
the reduction map gives an isomorphism of Galois modules
\begin{eqnarray}\label{redisom}
A[\ell^n]^\Ip \isom \NerA_p(\overline{\bf F}_p)[\ell^n]
\end{eqnarray}
(e.g., see~\cite[p.~495]{serre-tate}).

\begin{prop}
We have isomorphisms
$H^0(\Q_p, \Wl^\Ip) \isom \NerA_p(\F_p)[\ell^\infty]$ and
$H^0(\Qp, \Vl^\Ip/\Tl^\Ip) \isom \NerA_p^0(\F_p)[\ell^\infty]$.
\end{prop}

\begin{proof}
By lemma~\ref{lemma:qlzl}, $\Wl \isom A[\ell^\infty]$,
and from its proof, we see that this isomorphism respects
the action of the absolute Galois group of~$\Q_p$.
Using~(\ref{redisom}), we get 
$\Wl^\Ip = A[\ell^\infty]^\Ip \isom \NerA_p(\overline{\bf F}_p)[\ell^\infty]$.
Hence $H^0(\Q_p, \Wl^\Ip) = H^0( \Q_p / \Ip, \Wl^\Ip)
\isom (\NerA_p(\overline{\bf F}_p)[\ell^\infty])^\Frobp
= \NerA_p(\F_p)[\ell^\infty]$, which gives the first
isomorphism in the Proposition.

Next, note that since $\NerA_p^0$ is connected, the
system $\NerA_p^0[\ell^n]$ is surjective. Moreover,
since $\NerA_p^0$ is of finite index in $\NerA_p$,
by the construction of inverse limits, we have
$\Tl \NerA_p = \Tl \NerA_p^0$ (the only infinitely $\ell$-divisible
points in $\NerA_p(\overline{\bf F}_p)$ are
those coming from $\NerA_p^0(\overline{\bf F}_p)$).
Thus from~(\ref{redisom}), we have
an isomorphism $(\Tl A)^\Ip = \Tl (A^\Ip) \isom \Tl \NerA_p
= \Tl \NerA_p^0$. 
Thus 
\begin{tabbing}
\= $H^0(\Qp, \Vl^\Ip/\Tl^\Ip) = H^0(\Qp, \Tl^\Ip \tensor \Ql/\Zl)$ \\
\> $=  H^0(\Qp/\Ip, (\Tl A)^\Ip \tensor \Ql/\Zl)
\isom H^0(\Qp/\Ip, (\Tl \NerA_p^0)^\Ip \tensor \Ql/\Zl)$ \\
\> $= (\NerA_p^0(\overline{\bf F}_p)[\ell^\infty])^\Frobp
= \NerA_p^0(\F_p)[\ell^\infty]$.
\end{tabbing}
\end{proof}

Thus it follows that 
$\ord_\ell(c_p(1)) = \# H^0(\Qp, \Wl^\Ip) - 
\# H^0(\Qp, \Vl^\Ip/ \Tl^\Ip)
= \ord_\ell \NerA_p(\F_p) - \ord_\ell \NerA^0_p(\F_p)
= \ord_\ell (c_p(A))$, as was to be shown.

We next indicate how the definition in~(\ref{dswcp})
arises naturally from the formulation
of the Bloch-Kato conjecture~\cite{bloch-kato}.
Following 
the discussion on p.26-30 of~\cite{flach:degree},
if $c_p$ denotes the Tamagawa measure of~$A(\Ql)$ defined
in~\cite{bloch-kato}, then 
\begin{eqnarray} \label{bkcp}
\ord_\ell(c_p) = \# H^0(\Qp, \Wl) - \ord_\ell(P_p(p^{-1})),
\end{eqnarray}
where 
$P_p(p^{-s}) = \det(1 - \Frobp^{-1} p^{-s} 
\mid H^1_{\rm et}(A \tensor \Qbar, \Q_\ell))$
is the usual Euler factor at~$p$.
To see that the prime-to-$p$ parts of $c_p$ and $c_p(1)$ coincide,
one only needs the following Proposition, whose proof is skipped
in~\cite{flach:degree}.

\begin{prop}
$\ord_\ell(P_p(p^{-1})) = \# H^0(\Ql, \Vl^\Ip/\Tl^\Ip)$.
\end{prop}

\begin{proof}
We have the following isomorphisms of Galois modules:
$H^1_{\rm et}(A \tensor \Qbar, \Q_\ell) \isom \Hom(\Vl, \Ql)
\isom \Vl(-1)$, where the second isomorphism comes from 
the Weil pairing combined with a polarization map on~$A^\vee$ 
(thus the isomorphism may not hold over~$\Zl$). Hence
\begin{eqnarray} \label{eqn:1}
P_p(p^{-1})  & = &
\det(1 - \Frobp^{-1} p^{-1} \mid H^1_{\rm et}(A \tensor \Qbar, \Q_\ell)) 
\nonumber \\
& = & \det(1 - \Frobp^{-1} p^{-1} \mid \Vl(-1) ) \nonumber \\
& =& \det(1 - \Frobp^{-1} \mid \Vl).
\end{eqnarray}

Now there is a $\Zl$-basis 
$(e_1, \ldots, e_n)$ of~$\Tl^\Ip$ and 
$a_1, \ldots, a_n \in \Zl$ such that
$(a_1 e_1, \ldots, a_n e_n)$ is a basis for
the submodule $(1 - \Frobp^{-1}) \Tl^\Ip$
(see, e.g., \cite[Thm.~III.7.8]{lang:algebra}).
Then $\# \Tl^\Ip/(1 - \Frobp^{-1}) \Tl^\Ip = \ord_\ell(\prod_{i=1}^n a_i)$.
The endomorphism $(1 - \Frobp^{-1})$ may map~$\Tl^\Ip$ to
some basis of $(1 - \Frobp^{-1})\Tl^\Ip$ other than 
$(a_1 e_1, \ldots, a_n e_n)$,
but the change of basis matrix has determinant invertible in~$\Zl$.
Thus 
\begin{eqnarray} \label{eqn:2}
\ord_\ell(\det(1 - \Frobp^{-1} \mid \Vl)) &  = &
\ord_\ell(\prod_{i=1}^n a_i) \nonumber \\
& = & \# \Tl^\Ip/(1 - \Frobp^{-1}) \Tl^\Ip.
\end{eqnarray}
Next we need the following lemma:

\begin{lem} \label{lem:3}
The groups 
$ (\Vl^\Ip/\Tl^\Ip)^{\Frobp^{-1}}$ and 
$\Tl^\Ip/(1 - \Frobp^{-1}) \Tl^\Ip$ are isomorphic.
\end{lem}

\begin{proof}
If $v \in \Vl^\Ip$ is such that
$v + \Tl^\Ip \in (\Vl^\Ip/\Tl^\Ip)^{\Frobp^{-1}} $,
then $\Frobp^{-1} v - v \in \Tl^\Ip$.
This gives us a homomorphism $\phi: (\Vl^\Ip/\Tl^\Ip)^{\Frobp^{-1}} \ra 
(\Vl^\Ip/\Tl^\Ip)^{\Frobp^{-1}}$, which we will show is an isomorphism.
 
For simplicity, let $F = \Frobp^{-1}, T = (\Tl A)^\Ip$ and $V = (\Vl A)^\Ip$.
By~(\ref{redisom}), $T \isom \Tl \NerA_p(\overline{\bf F}_p)$; hence
$T^F = \Tl \NerA_p(\F_p)$. Thus $T^F$ is trivial,
and hence so is $V^F$.
We have an exact sequence
$$0 \ra V^F \ra V \stackrel{(1 - F)}{\longrightarrow} V \ra V/(1-F)V \ra 0.$$
Since $V^F$ is trivial, this shows that 
$V/(1-F)V$ is of dimension zero. Hence
$T/(1 - F)T$ is torsion.
Thus if $t\in T$, then there is an integer~$n$ such that
$nt \in (1 - F)T$; so there is a $t' \in T$ such
that $nt = (1-F) t'$. Then $v = t'/n \in V$ maps to~$t$.
Thus $\phi$ is surjective. Suppose $v \in V$ is such that
$\phi(v) = 0$. Then $F v - v = F t - t$, for 
some $t \in T$, so $v-t$ is fixed by $F$. 
But $V^F = 0$, so $v = t \in T$.
Thus $\phi$ is injective, and hence an isomorphism.
\end{proof}

By~(\ref{eqn:1}), (\ref{eqn:2}),  and Lemma~\ref{lem:3},
\begin{eqnarray*}
\ord_\ell(P_p(p^{-1})) = \# (\Vl/\Tl)^{\Frobp^{-1}}   
= \# (\Vl/\Tl)^\Frobp   
= \# H^0(\Ql, \Vl^\Ip/\Tl^\Ip), 
\end{eqnarray*}
which proves
the proposition.
\end{proof}

\providecommand{\bysame}{\leavevmode\hbox to3em{\hrulefill}\thinspace}
\providecommand{\MR}{\relax\ifhmode\unskip\space\fi MR }
\providecommand{\MRhref}[2]{%
  \href{http://www.ams.org/mathscinet-getitem?mr=#1}{#2}
}

\bibliographystyle{amsalpha}         

\begin{thebibliography}{DSW03}

\bibitem[Aga99]{agashe:invis}
A.~Agashe, \emph{On invisible elements of the {T}ate-{S}hafarevich group}, C.
  R. Acad. Sci. Paris S\'er. I Math. \textbf{328} (1999), no.~5, 369--374.

\bibitem[Aga00]{agashe:phd}
\bysame, \emph{The {B}irch and {S}winnerton-{D}yer formula for modular abelian
  varieties of analytic rank zero}, Ph.D. thesis, University of California,
  Berkeley (2000), available at \hfill \\ {\tt http://www.math.fsu.edu/\~{
  }agashe/math.html}.

\bibitem[Aga07]{ag:lett}
\bysame, \emph{Letter to {W}.~{S}tein}, available at \hfill \\ {\tt
  http://www.math.fsu.edu/\~{ }agashe/math.html}.

\bibitem[Aga08]{ag:congnum}
\bysame, \emph{The modular number, congruence number, and multiplicity one},
  preprint (2008), available at \hfill \\ {\tt http://www.math.fsu.edu/\~{
  }agashe/math.html}.

\bibitem[ARS06]{agst:manin}
Amod Agashe, Kenneth Ribet, and William~A. Stein, \emph{The {M}anin constant},
  Pure Appl. Math. Q. \textbf{2} (2006), no.~2, 617--636.

\bibitem[ARS07]{ars:moddeg}
A.~Agashe, K.~Ribet, and W.\thinspace{}A. Stein, \emph{{T}he modular degree,
  congruence primes, and multiplicity one}, preprint (2007), available at
  \hfill \\ {\sf http://www.math.fsu.edu/\~{ }agashe/moddeg3.html}.

\bibitem[AS02]{agst:vis}
Amod Agashe and William Stein, \emph{Visibility of {S}hafarevich-{T}ate groups
  of abelian varieties}, J. Number Theory \textbf{97} (2002), no.~1, 171--185.

\bibitem[AS05]{agst:bsd}
\bysame, \emph{Visible evidence for the {B}irch and {S}winnerton-{D}yer
  conjecture for modular abelian varieties of analytic rank zero}, Math. Comp.
  \textbf{74} (2005), no.~249, 455--484 (electronic), With an appendix by J.\
  Cremona and B.\ Mazur.

\bibitem[BK90]{bloch-kato}
S.~Bloch and K.~Kato, \emph{\protect{${L}$}-functions and \protect{T}amagawa
  numbers of motives}, The Grothendieck Festschrift, Vol. \protect{I},
  Birkh\"auser Boston, Boston, MA, 1990, pp.~333--400.

\bibitem[Car86]{carayol:hilbert}
H.~Carayol, \emph{Sur les \protect{repr\'esentations} \protect{$\ell$-adiques}
  \protect{associ\'ees} aux formes modulaires de \protect{Hilbert}}, Ann.
  scient. \protect{\'Ec.} Norm. Sup., \protect{$4^{\rm eb}$ s\'erie}
  \textbf{19} (1986), 409--468.

\bibitem[CM00]{cremona-mazur}
J.\thinspace{}E. Cremona and B.~Mazur, \emph{Visualizing elements in the
  {S}hafarevich-{T}ate group}, Experiment. Math. \textbf{9} (2000), no.~1,
  13--28. \MR{1 758 797}

\bibitem[DSW03]{dsw}
N.~Dummigan, W.~Stein, and M.~Watkins, \emph{Constructing elements in
  {S}hafarevich-{T}ate groups of modular motives}, Number theory and algebraic
  geometry, London Math. Soc. Lecture Note Ser., vol. 303, Cambridge Univ.
  Press, Cambridge, 2003, pp.~91--118.

\bibitem[Dum04]{dummigan:level}
N.~Dummigan, \emph{Level-lowering for higher congruences of modular forms},
  preprint (2004).

\bibitem[Eme03]{emerton:optimal}
Matthew Emerton, \emph{Optimal quotients of modular {J}acobians}, Math. Ann.
  \textbf{327} (2003), no.~3, 429--458.

\bibitem[Fla93]{flach:degree}
Matthias Flach, \emph{On the degree of modular parametrizations}, S\'eminaire
  de Th\'eorie des Nombres, Paris, 1991--92, Progr. Math., vol. 116,
  Birkh\"auser Boston, Boston, MA, 1993, pp.~23--36.

\bibitem[Gha02]{eki:congr}
Eknath Ghate, \emph{An introduction to congruences between modular forms},
  Currents trends in number theory (Allahabad, 2000), Hindustan Book Agency,
  New Delhi, 2002, pp.~39--58.

\bibitem[Gro72]{grothendieck:monodromie}
A.~Grothendieck, \emph{Groupes de monodromie en g\'eom\'etrie alg\'ebrique.
  \protect{I}}, Springer-Verlag, Berlin, 1972, S\'eminaire de G\'eom\'etrie
  Alg\'ebrique du Bois-Marie 1967--1969 (SGA 7 I), Dirig\'e par A.
  Grothendieck. Avec la collaboration de M. Raynaud et D. S. Rim, Lecture Notes
  in Mathematics, Vol. 288.

\bibitem[JS07]{jetchev-stein}
Dimitar~P. Jetchev and William~A. Stein, \emph{Visibility of the
  {S}hafarevich-{T}ate group at higher level}, Doc. Math. \textbf{12} (2007),
  673--696.

\bibitem[KL89]{kollog:finiteness}
V.\thinspace{}A. Kolyvagin and D.\thinspace{}Y. Logachev, \emph{Finiteness of
  the \protect{S}hafarevich-\protect{T}ate group and the group of rational
  points for some modular abelian varieties}, Algebra i Analiz \textbf{1}
  (1989), no.~5, 171--196.

\bibitem[Lan93]{lang:algebra}
S.~Lang, \emph{Algebra}, third ed., Addison-Wesley Publishing Co., Reading,
  Mass., 1993.

\bibitem[Lan95]{lang:modular}
\bysame, \emph{Introduction to modular forms}, Springer-Verlag, Berlin, 1995,
  With appendixes by D. Zagier and W. Feit, Corrected reprint of the 1976
  original.

\bibitem[Man72]{manin:parabolic}
J.\thinspace{}I. Manin, \emph{Parabolic points and zeta functions of modular
  curves}, Izv. Akad. Nauk SSSR Ser. Mat. \textbf{36} (1972), 19--66.

\bibitem[Maz77]{mazur:eisenstein}
B.~Mazur, \emph{Modular curves and the \protect{Eisenstein} ideal}, Inst.
  Hautes \'Etudes Sci. Publ. Math. (1977), no.~47, 33--186 (1978).

\bibitem[Par99]{parent}
Pierre Parent, \emph{Bornes effectives pour la torsion des courbes elliptiques
  sur les corps de nombres}, J. Reine Angew. Math. \textbf{506} (1999),
  85--116.

\bibitem[Rib83]{ribet:modp}
Kenneth~A. Ribet, \emph{Mod {$p$} {H}ecke operators and congruences between
  modular forms}, Invent. Math. \textbf{71} (1983), no.~1, 193--205.

\bibitem[Rub98]{rubin:eulerell}
Karl Rubin, \emph{Euler systems and modular elliptic curves}, Galois
  representations in arithmetic algebraic geometry (Durham, 1996), Cambridge
  Univ. Press, Cambridge, 1998, pp.~351--367.

\bibitem[Shi73]{shimura:factors}
G.~Shimura, \emph{On the factors of the jacobian variety of a modular function
  field}, J. Math. Soc. Japan \textbf{25} (1973), no.~3, 523--544.

\bibitem[Shi94]{shimura:intro}
\bysame, \emph{Introduction to the arithmetic theory of automorphic functions},
  Princeton University Press, Princeton, NJ, 1994, Reprint of the 1971
  original, Kan Memorial Lectures, 1.

\bibitem[Sil01]{silverberg:openque}
Alice Silverberg, \emph{Open questions in arithmetic algebraic geometry},
  Arithmetic algebraic geometry (Park City, UT, 1999), IAS/Park City Math.
  Ser., vol.~9, Amer. Math. Soc., Providence, RI, 2001, pp.~83--142.
  \MR{2002g:11073}

\bibitem[ST68]{serre-tate}
Jean-Pierre Serre and John Tate, \emph{Good reduction of abelian varieties},
  Ann. of Math. (2) \textbf{88} (1968), 492--517.

\bibitem[Ste]{stein:modformsdata}
W.\thinspace{}A. Stein, \emph{Modular forms database,\hfill\\ {\tt
http://modular.math.washington.edu/Tables/}}.

\bibitem[Vat99]{vatsal:canonical}
V.~Vatsal, \emph{Canonical periods and congruence formulae}, Duke Math. J.
  \textbf{98} (1999), no.~2, 397--419.

\end{thebibliography}

\providecommand{\href}[2]{#2}

\end{document}